\documentclass{article}
\usepackage{
multirow,authblk,
amssymb,amsmath,
amsthm,
bbm,
bm,
booktabs,mathtools
}
\usepackage[utf8]{inputenc}
\usepackage[ruled,procnumbered]{algorithm2e}
\usepackage{hyperref}
\hypersetup{colorlinks=true}
\DeclareUnicodeCharacter{00A0}{ }
\DeclareMathOperator*{\argmin}{arg\,min}
\DeclareMathOperator*{\spn}{span}
\DeclareMathOperator{\intrr}{\mathrm{int}}
\newcommand{\be}{\begin{equation}}
\newcommand{\ee}{\end{equation}}

\newcommand{\la}{\left\langle  }
\newcommand{\ra}{\right\rangle  }
\newcommand{\inv}{^{\raisebox{.2ex}{$\scriptscriptstyle-1$}}}
\newcommand{\mcb}[1]{\left\{#1\right\}}
\newcommand{\gbarv}{\bar{g}_{v}}
\newcommand{\mang}[1]{\left\langle  #1\right\rangle}
\newcommand{\mbk}[1]{\left(#1\right)}

\newcommand{\func}[1]{\mathrm {#1}\,}
\newcommand{\dom}{\func{dom}}
\newcommand{\epi}{\func{epi}}

\newcommand{\epri}{\epsilon\textrm{-}\func{ri}}
\renewcommand{\int}{\mathrm{int}}
\newcommand{\fa}{\forall \,}
\providecommand{\normt}[1]{\lVert#1\rVert}%two norm 
\newtheorem{asump}{Assumption}{\bf}{\it}
\theoremstyle{plain}
\newtheorem{theorem}{Theorem}
\newtheorem{lemma}{Lemma}
\newtheorem{proposition}{Proposition}
\theoremstyle{remark}
\newtheorem*{remark}{Remark}
\theoremstyle{definition}
\newtheorem{definition}{Definition}[section]
\newtheorem{corollary}{Corollary}
\begin{document}
\title{The U-Lagrangian of a prox-regular function}
\author{Shuai Liu\thanks{Email: \texttt{liushuai04235@gmail.com} Research supported by the Australian Research Council under Discovery Grant DP12100567.}}
\author{Andrew Eberhard\thanks{Email: \texttt{andy.eb@rmit.edu.au} Research supported by the Australian Research Council under Discovery Grant DP12100567.}}
\author{Yousong Luo\thanks{Email: \texttt{yluo@rmit.edu.au}}}
\affil{School of Mathematical and Geospatial Sciences, \\RMIT University, Melbourne, VIC 3001, Australia}
\date{}
\maketitle
\begin{abstract}
When restricted to a subspace, a nonsmooth function can be differentiable. It is known that for a nonsmooth convex function f and a point x, the Euclidean space can be decomposed into two subspaces: U, over which a special Lagrangian can be defined and has nice smooth properties and V, the orthogonal complement subspace of U. In this paper we generalize the definition of UV-decomposition and U-Lagrangian to the context of nonconvex functions, specifically that of a prox-regular function.
\end{abstract}
\textbf{Keywords} {UV-decomposition, \and U-Lagrangian, \and Prox-regular functions, \and Fast track, \and partly smooth}
\section{Introduction}
When studying the second order derivative of a nonsmooth function \(f\), one major difficulty is that the first-order approximation is not linear. 
The study of U-Lagrangian and UV-decomposition tries to overcome this difficulty by restricting the function to a subspace \(\mathcal{U}\) over which the function is actually differentiable. 
Hence the second-order expansion of \(f\) only needs to be defined along directions in \(\mathcal{U}\). The authors of \cite{lemarechal2000U-Lagrangian} developed the UV-decomposition and U-Lagrangian for a convex function; for instance \cite{MR2515798} studies
the minimax case. 
They showed that the U-Lagrangian is differentiable and a second-order expansion of \(f\) along directions in \(\mathcal{U}\) exists provided that the Hessian of the U-Lagrangian exists. 
The UV theory has been applied to the development of more efficient numerical algorithms such as in \cite{miller2005newton} and \cite{mifflin2005algorithm}, where approximated Newton steps in the \(\mathcal{U}\) space are made to help achieve superlinear convergence. 
Moreover, the objects associated with \(UV\)-decomposition can be easily approximated for functions with special structures such as the composition of a positively homogeneous convex function and a smooth mapping \cite{sagastizabal2013composite} and finite max functions \cite{hare2014numerical}. 
A subsmooth structure that is closely related to U-Lagrangian is \emph{fast track}\cite{mifflin2002proximal}. 
Roughly speaking, a fast track is a trajectory on which a certain second-order expansion of the underlying function can be obtained. 
Another related notion is partial smoothness defined in \cite{Lewis2002active} and it means over a smooth manifold the underlying function is smooth, regular, and has continuous first order derivative mapping. 
In \cite{hare2006functions} it is proved that fast track and partial smoothness are equivalent concepts under convexity. 
While most of the applications of \(UV\) theory are for solving convex optimization problems, theories in the nonconvex context have also been explored \cite{Mifflin2004UV,mifflin2003primal}. The quadratic sub-Lagrangian (QSL) \cite{hare2001quadratic} extends the U-Lagrangian to a type of nonconvex functions (called prox-regular functions) by adding a quadratic term to the infimand of the original U-Lagrangian. 
However, a strong quadratic growth condition is needed for QSL. 
In this paper we generalize the \(\mathcal{U}\)-Lagrangian to 
prox-regular functions from a different aspect. Instead of adding a quadratic term, we define the \(\mathcal{U}\)-Lagrangian "locally" because the prox-regularity is a local property. With this definition, no quadratic growth condition is needed. In addition, we can show that under the new U-Lagrangian, fast track and partial smoothness are almost equivalent for prox-regular functions.  

In this paper we use the following notations. The projection mapping onto a set $S$ is $P_S(x)$. The limiting normal of a set $C$ is $N_C$. The limiting subdifferential of a function $f$ is $\partial f(x)$. The indicator function of set $C$ is $\delta_C(x)$.
Denote $\bar{R}$ the extended real numbers. A set valued mapping \(S\) of two sets \(X\) and \(U\) is denoted by \(S\colon X\rightrightarrows U\). The set of all positive real numbers is \(\mathbb{R}_+\).
The closed ball in \(\mathbb{R}^n\) centered at \(\bar{x}\) with radius \(\epsilon\) is \(B(\bar{x},\epsilon)\). If \(E\) is a subspace of \(\mathbb{R}^n\) then \(B_E(w,\epsilon)\) is a closed ball in \(E\) centered at \(w\in E\), i.e. $B_{E}(w,\epsilon)\mathrel{\mathop:}=\{v\in E\colon\normt{v-w}\leq \epsilon\}$.
\section{Preliminaries}
In this section we provide some fundamental tools used in variational analysis and nonsmooth optimization. 
%For any $x\in \mathbb{R}^n$
%we define its two components associated with the decomposition via,
%\begin{equation}
%x_u\mathrel{\mathop:}=P_{\mathcal{U}}(x),\qquad x_v\mathrel{\mathop:}=P_{\mathcal{V}}(x).
%\end{equation}     
%%Let  be a proper l.s.c.  and  be $f$. 
%We define $B_{\mathcal{V}}(w,\epsilon)\mathrel{\mathop:}=\{v\in\mathcal{V}\colon\normt{v-w}\leq \epsilon\}$ for $w\in\mathcal{V}$ and $B_{\mathcal{U}}(w,\epsilon)$ is similarly defined. 
%We introduce the \emph{epsilon relative interior} as $\epsilon\text{-}\ri\partial f(x)\mathrel{\mathop:}=\{g^\circ\colon g^\circ+B_{\mathcal{V}}(0,\epsilon)\subset\partial f(x)\}$. 
\begin{definition}
A set $C$ is called prox-regular at $\bar{x}$ for $\bar{w}$, where $\bar{x}\in C$ and $\bar{w}\in N_C(\bar{x})$, if $\delta_C$ is prox-regular at $\bar{x}$ for $\bar{w}$. 
It is called prox-regular at \(\bar{x}\) when this is true for all \(\bar{w}\in N_C(\bar{x})\).
\end{definition}
\begin{proposition}[13.31 of \cite{Rockafellar1998}]\label{7prop:0}
For a set \(C\subset\mathbb{R}^n\) and a point \(\bar{x}\in \mathbb{R}^n\) the prox-regularity of \(C\) at \(\bar{x}\) for \(\bar{w}\) can be equivalently characterized by the following statement:\\
$C$ is locally closed at $\bar{x}$ with $\bar{w}\in N_C(\bar{x})$ and there exist \(\epsilon>0\) and \(\rho\geq 0\) such that 
\begin{equation}\label{7med:3}
\mang{w,x'-x}\leq\frac{1}{2}\rho\normt{x'-x}^2,\ \forall \,   x'\in C\cap B(\bar{x},\epsilon)
\end{equation}    
whenever 
\begin{equation}\label{7med:4}
x\in \intrr B\mbk{\bar{x},\epsilon}\text{ and } w\in N_C(x)\cap \intrr B\mbk{\bar{w},\epsilon}. 
\end{equation}
\end{proposition}
\begin{proposition}\label{7prop:1}
 %Let $P\colon\mathbb{R}^n\times\mathbb{R}^m\rightarrow\mathbb{R}^n$ be defined as $P(x,y)=x$ and $D\mathrel{\mathop:}=P(C)$ for a closed set $C\subset\mathbb{R}^n\times\mathbb{R}^m$.
With \(\mathbb{R}^n\) expressed as \(\mathbb{R}^{n_1}\times\mathbb{R}^{n_2}\), write \(x\in\mathbb{R}^n\) as \(\mbk{x_1,x_2}\) with components \(x_i\in\mathbb{R}^{n_i}\). 
Suppose \(C=D\times E\) for closed sets \(D\in \mathbb{R}^{n_1}\) and \(E\in \mathbb{R}^{n_2}\). 
If $C$ is prox-regular at $(\bar{x},\bar{y})$ for $(\bar{w},\bar{z})$ with respect to $\epsilon$ and $\rho$, 
 then $D$ is prox-regular at $\bar{x}$ for $\bar{w}$ with respect to $\epsilon$ and $\rho$.
 \end{proposition}
 \begin{proof}
 From 13.31 of \cite{Rockafellar1998} we know that if $C$ is prox-regular at $(\bar{x},\bar{y})$ for $(\bar{w},\bar{z})$ with respect to $\epsilon$ and $\rho$, then $C$ is locally closed at $(\bar{x},\bar{y})$ with $(\bar{w},\bar{z})\in N_C(\bar{x},\bar{y})$ and 
 \begin{equation}\label{7med:3}
 \mang{(w,z),(x',y')-(x,y)}\leq\frac{1}{2}\rho\normt{(x',y')-(x,y)}^2,\ \forall \,  (x',y')\in C\cap B((\bar{x},\bar{y}),\epsilon)
 \end{equation}    
 whenever 
 \begin{equation}\label{7med:4}
 (w,z)\in N_C(x,y),\ \normt{(w,z)-(\bar{w},\bar{z})}<\epsilon\text{ and }\normt{(x,y)-(\bar{x},\bar{y})}<\epsilon.
 \end{equation}
 We now show $D$ is prox-regular at $(\bar{x},\bar{y})$ for $(\bar{w},\bar{z})$ with respect to $\epsilon$ and $\rho$ by verifying 13.31 of \cite{Rockafellar1998}. 
 Obviously $D$ is locally closed at $\bar{x}$ because $D$ is a closed set and $C$ is locally closed at $(\bar{x},\bar{y})$. 
 First, by 6.41 of \cite{Rockafellar1998} we have $\bar{w}\in N_D(\bar{x})$ and $w\in N_D(x)$. In \eqref{7med:3} and \eqref{7med:4} we can take $y'=y=\bar{y}$ and $z=\bar{z}$ to obtain
 \[
 \mang{(w,\bar{z}),(x',\bar{y})-(x,\bar{y})}\leq\frac{1}{2}\rho\normt{(x',\bar{y})-(x,\bar{y})}^2,\ \forall \,  (x',\bar{y})\in C\cap B((\bar{x},\bar{y}),\epsilon)
 \]
 whenever
 \[
 (w,\bar{z})\in N_C(x,\bar{y}),\ \normt{(w,\bar{z})-(\bar{w},\bar{z})}<\epsilon\text{ and }\normt{(x,\bar{y})-(\bar{x},\bar{y})}<\epsilon.
 \]
 This verifies that there exist $\epsilon>0$ and $\rho\geq 0$ such that $\mang{w,x'-x}\leq \frac{1}{2}\rho\normt{x'-x}^2$ for all $x'\in D
 \cap B(\bar{x},\epsilon)$ when $w\in N_C(x)$, $\normt{w-\bar{w}}<\epsilon$ and $\normt{x-\bar{x}}<\epsilon$.  
 \end{proof}
 \begin{lemma}\label{7lem:0}
Given \(\bar{y}\in\mathbb{R}^n\) and \(\alpha\in\mathbb{R}_+\), for any \(\beta\in]0,\alpha[\) 
one has \[\intrr B\mbk{y,\alpha-\beta}\subset \int B\mbk{\bar{y},\alpha} \mathrm{ and  }B\mbk{y,\alpha-\beta}\subset  B\mbk{\bar{y},\alpha}, \ \fa y\in B\mbk{\bar{y},\beta}.\]
\end{lemma}
\begin{proof}
Taking \(\beta\in]0,\alpha[\), \(y\in B\mbk{\bar{y},\beta}\) and \(z\in \intrr B\mbk{y,\alpha-\beta}\), one has \(\normt{z-\bar{y}}=\normt{z-y+y-\bar{y}}\leq \normt{z-y}+\normt{y-\bar{y}}<\alpha-\beta+\beta=\alpha\) and thus \(z\in \int B\mbk{\bar{y},\alpha}\). The second part of the conclusion can be proved similarly.
\end{proof}
\begin{proposition}
Let \(C\) be a closed set in \(\mathbb{R}^n\). 
If \(C\) is prox-regular at \(\bar{x}\) for \(\bar{v}\) with respect to \(\bar{\epsilon}\) and \(\rho\), 
then for all \(\beta\in]0,\bar{\epsilon}[\), \(\tilde{x}\in B\mbk{\bar{x},\beta}\), 
and \(\tilde{v}\in N_C\mbk{\tilde{x}}\cap B\mbk{\bar{v},\beta}\), one has \(C\) is prox-regular at \(\tilde{x}\) for 
\(\tilde{v}\) with respect to \(\bar{\epsilon}-\beta\) and \(\rho\).
\end{proposition}
\begin{proof}
For all \(\beta\in]0,\bar{\epsilon}[\), \(\tilde{x}\in B\mbk{\bar{x},\beta}\), 
and \(\tilde{v}\in N_C\mbk{\tilde{x}}\cap B\mbk{\bar{v},\beta}\), by Proposition \ref{7prop:0} we need to prove \begin{equation}\label{7med:p21}
\mang{w,x'-x}\leq\frac{1}{2}\rho\normt{x'-x}^2,\ \forall \,   x'\in C\cap B(\tilde{x},\bar\epsilon-\beta)
\end{equation}    
whenever 
\begin{equation}\label{7med:p22}
x\in \intrr B\mbk{\tilde{x},\bar\epsilon-\beta}\text{ and } w\in N_C(x)\cap \intrr B\mbk{\tilde{w},\bar\epsilon-\beta}. 
\end{equation}
Applying Lemma \ref{7lem:0} to \(\bar{x}\) and \(\bar{\epsilon}\), we verify that 
\begin{gather}
\intrr B\mbk{\tilde{x},\bar\epsilon-\beta}\subset \int B\mbk{\bar{x},\bar\epsilon},\ \intrr B\mbk{\tilde{v},\bar\epsilon-\beta}\subset \int B\mbk{\bar{v},\bar\epsilon}\mathrm{ and }\label{7eq:prop24}\\
B\mbk{\tilde{x},\bar\epsilon-\beta}\subset B\mbk{\bar{x},\bar\epsilon}.\label{7eq:prop27}
\end{gather}
If \eqref{7med:p22} is true then together with \eqref{7eq:prop24} it implies 
\begin{equation}\label{7eq:prop25}
x\in \int B\mbk{\bar{x},\bar\epsilon}\text{ and } w\in N_C(x)\cap \int B\mbk{\bar{v},\bar\epsilon}. 
\end{equation}
Under \eqref{7eq:prop25}, the prox-regularity of \(\bar{x}\) at \(\bar{v}\) reveals 
\begin{equation}\label{7eq:pro26}
\mang{w,x''-x}\leq\frac{1}{2}\rho\normt{x''-x}^2,\ \forall \,   x''\in C\cap B(\bar{x},\bar\epsilon).
\end{equation}
Then combining \eqref{7eq:pro26} and \eqref{7eq:prop27} we get \eqref{7med:p21}.
\end{proof}
\begin{definition}
We say a function \(f\colon\mathbb{R}^n\rightarrow \bar{\mathbb{R}}\) is properly prox-regular at \(\bar{x}\) for \(\bar{w}\) if \(f\) is prox-regular at \(\bar{x}\) for \(\bar{w}\) and \(\epi f\) is prox-regular at \((\bar{x},f(\bar{x}))\) for \((\bar{w},-1)\). 

If the conditions hold for all \(g\in\partial f(\bar{x})\) and all \(({g},-1)\in N_{\epi f}(\bar{x},f(\bar{x}))\) then we say \(f\) is properly prox-regular at \(\bar{x}\).
\end{definition}
\begin{proposition}\label{7prop:3}
(i) If a set $C\subset\mathbb{R}^n$ is prox-regular at $\bar{x}$ then $C$ is Clarke regular at $\bar{x}$.

(ii) If a function \(f\) is properly prox-regular at \(\bar{x}\) then \(f\) is subdifferentially regular at \(\bar{x}\).
\end{proposition}
\begin{proof}
If $C\subset\mathbb{R}^n$ is prox-regular at $\bar{x}$ then $C$ is locally closed at $\bar{x}$ and for each $\bar{v}\in N_C(\bar{x})$ there exist $\epsilon>0$ and $\rho\geq 0$ such that $\mang{v,x'-x}\leq\frac{1}{2}\rho\normt{x'-x}^2$ for all $x'\in C\cap \mathbb{B}(\bar{x},\epsilon)$ when $v\in N_C(x)$, $\normt{v-\bar{v}}<\epsilon$ and $\normt{x-\bar{x}}<\epsilon$. 
For each $\bar{v}\in N_C(\bar{x})$ we take $v=\bar{v}$ and $x=\bar{x}$ to obtain $\mang{\bar{v},x'-\bar{x}}\leq\frac{1}{2}\rho\normt{x'-\bar{x}}^2$ for all $x'\in C\cap \mathbb{B}(\bar{x},\epsilon)$. 
This means $\limsup\limits_{\mathclap{\substack{x'\underset{\mathclap{C}}{\rightarrow}\bar{x}\\x'\not=\bar{x}} }}\frac{\mang{\bar{v},x'-\bar{x}}}{\normt{x'-\bar{x}}}\leq 0$. 
Therefore $\bar{v}\in\hat{N}_C\mbk{\bar{x}}$ and $C$ is Clarke regular at \(\bar{x}\). Conclusion (ii) is immediate from (i) and the definition of subdifferential regularity.
\end{proof}
The following proposition is taken from 1.107 of \cite{Mordukhovich2006}.
\begin{proposition}\label{7pro:5}
Given an arbitrary function \(\varphi\colon\mathbb{R}^n\mapsto\bar{\mathbb{R}}\) finite at \(\bar{x}\), the following hold:\\
\indent (i) For any \(\phi\colon\mathbb{R}^n\mapsto\bar{\mathbb{R}}\) Fr\'{e}chet differentiable at \(\bar{x}\) one has 
\begin{equation}
\hat\partial \mbk{\phi+\varphi}\mbk{\bar{x}}=\nabla\phi\mbk{\bar{x}}+\hat\partial\varphi\mbk{\bar{x}}.
\end{equation}
\indent (ii) For any \(\phi\colon\mathbb{R}^n\mapsto\bar{\mathbb{R}}\) strictly differentiable at \(\bar{x}\) one has 
\begin{equation}
\partial \mbk{\phi+\varphi}\mbk{\bar{x}}=\nabla\phi\mbk{\bar{x}}+\partial\varphi\mbk{\bar{x}}.
\end{equation}
\end{proposition}
\section{U-Lagrangian}
We begin this section with a very generic definition of the subspaces we will use. The official UV-decomposition will be defined in the next section. 
\begin{definition}\label{7def:1}
Given a point \(\bar{x}\), let \(\mathcal{V}(\bar{x})\) be a 
subspace of \(\mathbb{R}^n\) 
such that the set \(\mcb{ g^{\circ}\colon g^{\circ}+ \intrr B_{\mathcal{V}}(0,\epsilon)\subset\partial f(\bar{x}) }\) is not empty. 
We define \(\mathcal{U}(\bar{x})\mathrel{\mathop:}=\mathcal{V}(\bar{x})^{\bot}\) 
so that for any $x\in \mathbb{R}^n$
we have its two components associated with the 
decomposition via,
\begin{equation}
x_u\mathrel{\mathop:}=P_{\mathcal{U}(\bar{x})}(x),\qquad x_v\mathrel{\mathop:}=P_{\mathcal{V}(\bar{x})}(x).
\end{equation}   
To simplify notation we define \(D_{\epsilon}f\mathrel{\mathop:}=\{g^{\circ}\colon g^{\circ}+\intrr B_{\mathcal{V}}(0,\epsilon)\subset\partial f(\bar{x})\}\) and \(\mathcal{U}\mathrel{\mathop:}=\mathcal{U}(\bar{x})\) and \(\mathcal{V}\mathrel{\mathop:}=\mathcal{V}(\bar{x})\).
\end{definition}%Let $\mathcal{U}$ and $\mathcal{V}$ be any subspaces satisfying $\mathcal{U}\oplus\mathcal{V}=\mathbb{R}^n$. 

\begin{asump}\label{7as:1}
The function $f$ is proper, l.s.c. on $\mathbb{R}^n$ and properly prox-regular at $\bar{x}$ with respect to $\bar{\epsilon}$ and $\rho$.
\end{asump}
\begin{definition}\label{def:1}
Given $\epsilon>0$,
we take an arbitrary 
\(\bar{g}\in D_{\epsilon}f\)
%$\bar{g}\in\epri\partial f(\bar{x})$ 
and define the function $L_{\epsilon}$ as follows:
\begin{equation}\label{def:L}
\mathcal{U}\ni u\mapsto L_{\epsilon}(u;\bar{g}_{v})\mathrel{\mathop:}=\inf_{v\in B_{\mathcal{V}}(0,\epsilon)}\mcb{f(\bar{x}+u+v)-\mang{\bar{g}_{v},v}}.
\end{equation}
Associated with \eqref{def:L} we have the set of $\mathcal{V}$-space minimizers 
\begin{equation}
W(u;\gbarv)\mathrel{\mathop:}=\argmin_{v\in B_{\mathcal{V}}(0,\epsilon)} \mcb{f(\bar{x}+u+v)-\mang{\bar{g}_{v},v}}.
\end{equation}
\end{definition}
To simplify notation we let 
\begin{equation}
h(u,v)\mathrel{\mathop:}=f(\bar{x}+u+v)-\mang{\bar{g}_v,v}+\delta_{B_{\mathcal{V}}(0,\epsilon)}(v).
\end{equation}
\begin{theorem}\label{7th:2}
Suppose Assumption \ref{7as:1} holds, then 

(i) \(h(u,v)\) is proper, l.s.c. on \(\mathcal{U}\times\mathcal{V}\) and level bounded in $v$ locally uniformly in $u$;

(ii) \(\partial h(0,0)=\mcb{(g_u,g_v-\bar{g}_v)\colon g\in\partial f(\bar{x})}\) and \(h(u,v)\) is prox-regular at \((0,0)\) with respect to \(\bar{\epsilon}\) and \(\rho\);

(iii) $L_{\epsilon}$ is proper and l.s.c. on $\mathcal{U}$, and for each $u\in \dom L_{\epsilon}$ the set $W(u;\bar{g}_v)$ is nonempty and compact whereas $W(u;\bar{g}_v)=\emptyset$ when $u\not\in\dom L_{\epsilon}$;

(iv) For each \(s\in \partial L_{\epsilon}\mbk{u;\bar{g}_v} \) there exists \(\hat{v}\in W\mbk{u;\bar{g}_v}\) such that \(\mbk{s,0}\in\partial h\mbk{u,\hat{v}}\).
\end{theorem}
\begin{proof}
(i) We have $h$ is proper on $\mathcal{U}\times\mathcal{V}$ because $h(0,0)=f(\bar{x})$ is finite from prox-regularity of $f$ at $\bar{x}$. We also have $h$ is l.s.c. because $f$ is l.s.c. and $B_{\mathcal{V}}(0,\epsilon)$ is closed. We show $h(u,v)$ is
level-bounded in $v$ locally uniformly in $u$ by equivalently showing that the mapping $u\mapsto\mcb{v\colon h(u,v)\leq\alpha}$ is locally bounded for each $\alpha\in\mathbb{R}$ (see 5.17 of \cite{Rockafellar1998}). Let $S(u')$ be an arbitrary neighborhood of an arbitrary point $u'\in\mathcal{U}$, the set $\mcb{v\colon h(u,v)\leq \alpha,u\in S(u')}$ is clearly contained in $B_{\mathcal{V}}(0,\epsilon)$ for all $u'$ and $\alpha$. Hence $u\mapsto\mcb{v\colon h(u,v)\leq\alpha}$ is locally bounded for each $\alpha\in\mathbb{R}$.

(ii) Consider the function \(h_1\colon\mathcal{U}\times\mathcal{V}\mapsto\bar{\mathbb{R}}\) defined by \(h_1(u,v)=f(\bar{x}+u+v)\). The subdifferential of \(h_1\) is \(\partial h_1(u,v)=\mcb{(g_u,g_v)\colon g\in\partial f(\bar{x}+u+v)}\). 
From the definition of prox-regularity and the fact that \(f\) is prox-regular at \(\bar{x}\) we can easily verify by definition that \(h_1\) is prox-regular at \((0,0)\).   
We then write \(h(u,v)=h_1(u,v)+h_2(v)\) where \(h_2(v)=-\mang{\bar{g}_v,v}+\delta_{B_{\mathcal{V}}(0,\epsilon)}(v)\). 
We have \(\partial h_2(0)=\mcb{-\bar{g}_v}\). 
By Assumption \ref{7as:1} we have \(f\) is properly prox-regular at \(\bar{x}\). 
From Proposition \ref{7prop:3}(ii) we have \(f\) is subdifferentially regular at \(\bar{x}\). 
We can also verify that \(h_1\) is properly prox-regular at \((0,0)\) (as a straight forward application of Proposition \ref{7prop:0}) and hence subdifferentially regular there.  
Thus \(h\) is subdifferentially regular at \((0,0)\). From 10.9 of \cite{Rockafellar1998} we have \(\partial h(0,0)=\partial h_1(0,0)+\mcb{0,\partial h_2(0)}=\mcb{(g_u,g_v-\bar{g}_v)\colon g\in\partial f(\bar{x})}\). By 13.35 of \cite{Rockafellar1998} we have \(h\) is prox-regular at \((0,0)\).

(iii) We see $L_{\epsilon}(u;\bar{g}_v)=\inf_{v\in\mathcal{V}}\mcb{h(u,v)}$. By 1.17 of \cite{Rockafellar1998} it suffices to show (i).

(iv) Under conclusion (i) we can apply 10.13 of \cite{Rockafellar1998} to obtain\\ 
\(\partial L_{\epsilon}\mbk{u;\bar{g}_v}\subset\cup_{\hat{v}\in W\mbk{u;\bar{g}_v}}\mcb{w\colon\mbk{w,0}\in\partial h\mbk{u,\hat{v }}}\). 
Since \(s\in \partial L_{\epsilon}\mbk{u;\bar{g}_v}\), there exists \(\hat{v}\) such that \(\mbk{s,0}\in \partial h(u,\hat{v})\). 
\end{proof}
\begin{asump}\label{7as:2}
We assume that in Assumption \ref{7as:1}, $\rho\in ]0,2[$ and $\bar{\epsilon}>\epsilon$ where $\epsilon$ is introduced in Definition \ref{def:1}.
\end{asump}
\begin{definition}[Proximal subdifferential]\label{7def:1}
A vector \(g\) is called a proximal subgradient of a function \(f\colon\mathbb{R}^n\mapsto\bar{\mathbb{R}}\) at \(\bar{x}\in\dom f\) if there exist \(\epsilon>0\) and \(\rho>0\) such that 
\begin{equation}
f(x)\geq f(\bar{x})+\mang{v,x-\bar{x}}-\frac{\rho}{2}\normt{x-\bar{x}}^2\text{ when }\normt{x-\bar{x}}\leq\epsilon.
\end{equation}
The set of all proximal subgradients of \(f\) at \(\bar{x}\) is called the proximal subdifferential and is denoted by \(\partial_p f(\bar{x})\). If \(\bar{x}\not\in\dom f\) then \(\partial_p f(\bar{x})=\emptyset\).
\end{definition}

\begin{theorem}\label{7th:3}
Suppose Assumptions \ref{7as:1} and \ref{7as:2} hold.

(i) $L_{\epsilon}(0;\bar{g}_v)=f(\bar{x})$ and $W(0;\bar{g}_v)=\{0\}$;

(ii) \(\bar{g}_u\in\partial_p L_{\epsilon}(0;\bar{g}_v)\).

(iii) \(
L_{\epsilon}(u;\bar{g}_v)\geq f(\bar{x})+\mang{\bar{g}_u,u}-\frac{\rho}{2}\normt{u}^2,\ \forall \,   u\in B_{\mathcal{U}}(0,r),\) where \(r=\sqrt{\bar{\epsilon}-\epsilon}\).
\end{theorem}
\begin{proof}
Assumption \ref{7as:1} yields 
\begin{equation}\label{med:1}
f(x')\geq f(\bar{x})+\la g,x'-\bar{x}\ra-\frac{\rho}{2}\normt{x'-\bar{x}}^2\text{ for any }g\in\partial f(\bar{x})
\end{equation}
whenever $\normt{x'-\bar{x}}\leq\bar{\epsilon}$. For all $u\in B_{\mathcal{U}}(0,r)$ and $v\in B_{\mathcal{V}}(0,\epsilon)$, one has $\normt{u}^2+\normt{v}^2\in[0,\bar{\epsilon}^2[$. Consequently, \eqref{med:1} holds for $x'=\bar{x}+u+v$, i.e.
\begin{align}\label{7med:5}
f(\bar{x}+u+v)&\geq f(\bar{x})+\mang{g,u+v}-\frac{\rho}{2}(\normt{u}^2+\normt{v}^2)\quad\text{for any }g\in\partial f(\bar{x}). 
\end{align}
By the definition of $\epri\partial f(\bar{x})$, for any $v\in B_{\mathcal{V}}(0,\epsilon)$, there exists a $g'\in\partial f(\bar{x})$ such that $\bar{g}+v=g'$. 
In \eqref{7med:5} we can take $g=g'$ introduced above and get
\begin{align*}
f(\bar{x}+u+v) &\geq f(\bar{x})+\mang{\bar{g}+v,u+v}-\frac{\rho}{2}(\normt{u}^2+\normt{v}^2)\\
 & = f(\bar{x})+\mang{\bar{g}_u,u}-\frac{\rho}{2}\normt{u}^2+\mang{\bar{g}_v,v}+(1-\frac{\rho}{2})\normt{v}^2
\end{align*}
for all \(u\in B_{\mathcal{U}}(0,r)\) and all \(v\in B_{\mathcal{V}}(0,\epsilon)\). Subtracting \(\mang{\bar{g}_v,v}\) on both sides, we have 
\begin{align}\label{7med:2}
f(\bar{x}+u+v)- \mang{\bar{g}_v,v}\geq f(\bar{x})+\mang{\bar{g}_u,u}-\frac{\rho}{2}\normt{u}^2+(1-\frac{\rho}{2})\normt{v}^2.
\end{align}
By the definition of $L_{\epsilon}$, the fact that \(\rho\in]0,2[\) and \eqref{7med:2} we have 
\begin{equation}\label{med:3}
L_{\epsilon}(u;\bar{g}_v)\geq f(\bar{x})+\mang{\bar{g}_u,u}-\frac{\rho}{2}\normt{u}^2,\ \forall \,  u\in B_{\mathcal{U}}(0,r). 
\end{equation}
By definition we know $L_{\epsilon}(0;\bar{g}_v)\leq f(\bar{x}+0)-\mang{\bar{g}_v,0}=f(\bar{x})$. On the other hand, replacing $u$ in \eqref{med:3} by 0 yields $L_{\epsilon}(0;\bar{g}_v)\geq f(\bar{x})$. 
Thus $L_{\epsilon}(0;\bar{g}_v)=f(\bar{x})$. 
To show $W(0;\bar{g}_v)=\{0\}$, suppose for contradiction that there exists \(v'\in W(0;\bar{g}_v)\) but \(v'\not=0\). 
We apply \eqref{7med:2} to \(u=0\) and \(v=v'\) and get \[L_{\epsilon}(0;\bar{g}_v)=f(\bar{x}+v')-\mang{\bar{g}_v,v'}\geq f(\bar{x})+(1-\frac{\rho}{2})\normt{v'}^2>f(\bar{x})=L_{\epsilon}(0;\bar{g}_v).\] A contradiction.

(ii) Replacing \(u\) in \eqref{med:3} by 0 and \(f(\bar{x})\) by \(L_{\epsilon}(0;\bar{g}_v)\) we can see \(\bar{g}_u\in\partial_p L_{\epsilon}(0;\bar{g}_v)\). 
\end{proof}
Define function  \be\label{7:eq:20} \mathbb{R}^n\ni v\mapsto F(v;\bar{g}_v)\mathrel{\mathop:}=f(\bar{x}+v)-\mang{\bar{g}_v,v}+\delta_{B_{\mathcal{V}}(0,\epsilon)}(v),\ee where \(\bar{g}_v\in P_{\mathcal{V}}\mbk{\partial f\mbk{\bar{x}}} \) is a parameter. 

To simplify notation we  sometimes omit the parameter \(\bar{g}_v\) in \(F(v;\bar{g}_v)\) when it does not affect the understanding. Define the function  \(f_{\mathcal{V}}\colon \mathcal{V}\rightarrow\bar{R}\) as \be\label{7:eq:21}f_{\mathcal{V}}(v)=f\mbk{\bar{x}+v}.\ee
Define the function \(q\colon \mathbb{R}^n\rightarrow \bar{R}\) as 
\be\label{7:eq:22}
q(x)\mathrel{\mathop:}= f\mbk{\bar{x}+x}.
\ee
\begin{lemma}
\label{7:lem:u}
As \(f\) is subdifferentially regular at \(\bar{x}\), from the basic chain rule we have \(\partial f_{\mathcal{V}}(0))=P_{\mathcal{V}}\mbk{\partial f\mbk{\bar{x}}}\) and \(\partial q(0)=\partial f\mbk{\bar{x}}\).
\end{lemma}
\begin{proposition}\label{7:pro:5}
Suppose \(f\) satisfies Assumption \ref{7as:1}. The function \(F(v;\bar{g}_v)\) is %properly 
prox-regular at 0 for all \(\bar{g}_v\in P_{\mathcal{V}}\mbk{\partial f\mbk{\bar{x}}} \). %\textcolor{red}{seems that i don't need to add the "properly"!}
\end{proposition}
\begin{proof}
To show \(F\) is prox-regular at 0 we only need to show \(q\) is prox-regular at 0 as the function \(-\mang{\bar{g}_v,v}+\delta_{B_{\mathcal{V}}(0,\epsilon)}(v)\) is smooth around \(0\); see 13.35 in \cite{Rockafellar1998}. 
We can easily check that \(q(0)=f\mbk{\bar{x}}\) is finite and \(q(v)\) is locally l.s.c. at 0 from Assumption \ref{7as:1}; 
Also it follows from Lemma \ref{7:lem:u} that \(\partial q(0)=\partial f\mbk{\bar{x}}\).
The basic chain rule reveals \(\partial q(v)\subset \partial f\mbk{\bar{x}+v}\). 
Now we use the definition of prox-regularity to show \(q(v)\) is prox-regular at 0. 
For all \(\bar{g}\in \partial q(0)\), when \(s\in\partial q(v),\ \normt{s-\bar{g}}<\bar{\epsilon},\ \normt{v-0}<\bar{\epsilon},\ q(v)<q(0)+\bar{\epsilon} \), where \(\bar{\epsilon}\) is introduced in Assumption \ref{7as:1}, 
we have \(s\in \partial f\mbk{\bar{x}+v},\ \normt{s-\bar{g}}<\bar{\epsilon},\ \normt{\bar{x}+v-\bar{x}}<\bar{\epsilon},\ f\mbk{\bar{x}+v}<f\mbk{\bar{x}}+\bar{\epsilon}\). 
From the prox-regularity of \(f\) at \(\bar{x}\), for all \(\bar{g}\in \partial f\mbk{\bar{x}}\), we get 
\(f\mbk{x'}\geq f\mbk{\bar{x}+v}+\mang{s,x'-\mbk{\bar{x}+v}}-\frac{\rho}{2}\normt{x'-\mbk{\bar{x}+v}}^2,\ \fa x'\in B\mbk{\bar{x},\bar \epsilon}\). 
As there is a on-to-one correspondent between \(x'\in B\mbk{\bar{x},\bar \epsilon}\) and \(v'\in B\mbk{0,\bar{\epsilon}}\) such that \(x'=\bar{x}+v'\), it follows that 
\(f\mbk{\bar{x}+v'}\geq f\mbk{\bar{x}+v}+\mang{s,v'-v}-\frac{\rho}{2}\normt{v'-v}^2,\ \fa v'\in B\mbk{0,\bar{\epsilon}}\), i.e. 
\(q(v)\geq q(v)+\mang{s,v'-v}-\frac{\rho}{2}\normt{v'-v}^2, \ \fa v'\in B\mbk{0,\bar{\epsilon}}\). 
This finishes the proof of the prox-regularity of \(q \) and hence of \(F\).
\end{proof}
Here we investigate a special property of the function \(F\), tilt-stability, introduced in \cite{poliquin1998tilt}. 
\begin{definition}
A point \(\bar{x}\) is said to give a tilt-stable local minimum of the function \(
f\colon\mathbb{R}^n\mapsto\bar{\mathbb{R}
} 
\)
if \(f(\bar{x})\) is finite and there exists \(\delta\in\mathbb{R}_{+}\) such that the mapping 
\[
M\colon g\mapsto\argmin\limits_{\normt{x-\bar{x}}\leq \delta}\mcb{f(x)-f(\bar{x})-\mang{g,x-\bar{x}}}
\]
is sing-valued and Lipschitzian on some neighborhood of \(g=0\) with \(M(0)=\bar{x}\).
\end{definition}
\begin{proposition}\label{7:prop:6}
Under Assumptions \ref{7as:1} and \ref{7as:2}, \(0\) gives a tilt-stable local minimum of \(F(v;\bar{g}_v)\) for all \(\bar{g}_v\in P_{\mathcal{V}}\mbk{\partial f\mbk{\bar{x}}} \), where \(F\) is defined in \eqref{7:eq:20}. 
\end{proposition}
\begin{proof}
Let \(\bar{g}_v\) be an arbitrary element in \(P_{\mathcal{V}}\mbk{\partial f\mbk{\bar{x}}} \). 
We see \(F(0)=f(\bar{x})\) is finite from Assumption \ref{7as:1}. 
Then \(F(v)=f_{\mathcal{V}}(v)-\mang{\bar{g}_v,v}+\delta_{B_{\mathcal{V}}(0,\epsilon)}(v)\) for all \(v\in\mathcal{V}\), where \(f_{\mathcal{V}}(v)\) is defined in \eqref{7:eq:21}.  
We have \(\partial F(0)=\partial f_{\mathcal{V}}(0)-\bar{g}_v\) from Proposition \ref{7pro:5}.  
 and thus \(\partial F(0)=P_{\mathcal{V}}(\partial f(\bar{x}))-\bar{g}_v\) from Lemma \ref{7:lem:u}. 
Since \(\bar{g}_v\in P_{\mathcal{V}} \mbk{\partial f(\bar{x})}\) we have \(0\in \partial F(0)\). 
Additionally we have \(0\) is a local minimizer of \(F\) because
\[
F(0)=f(\bar{x})=L_{\epsilon}(0;\bar{g}_v)\leq f(\bar{x}+v)-\mang{\bar{g}_v,v}=F(v),\quad\forall \,  v\in B_{\mathcal{V}}(0,\epsilon)
\]
From \(\partial F(0)=\mcb{P_{\mathcal{V}}(\partial f(\bar{x}))-\bar{g}_v}\) we have \(s+\bar{g}_v\in P_{\mathcal{V}}\mbk{\partial f\mbk{\bar{x}}}\) for all \(s\in\partial F(0)\). 
Next we show that the mapping 
\[
M\colon s\mapsto\argmin\limits_{v\in B_{\mathcal{V}}(0,\epsilon)}\mcb{F(v)-F(0)-\mang{s,v}}
\]
is single-valued and Lipschitzian on some neighborhood of \(0\), particularly on the following set, \(E\mathrel{\mathop:}=\mcb{s\in\partial F(0)\colon s+\bar{g}_v\in P_{\mathcal{V}}\mbk{D_{\epsilon}f}}\). 
\begin{align*}
M(s)=\argmin\limits_{v\in B_{\mathcal{V}}(0,\epsilon)}\mcb{F(v)-F(0)-\mang{s,v}}
\\=\argmin\limits_{v\in B_{\mathcal{V}}(0,\epsilon)}\mcb{F(v)-\mang{s,v}}
\\=\argmin\limits_{v\in B_{\mathcal{V}}(0,\epsilon)}\mcb{f(\bar{x}+v)-\mang{\bar{g}_v
+s,v}}
\end{align*}
Because \(s+\bar{g}_v\in P_{\mathcal{V}}\mbk{D_{\epsilon}f}\) for all \(s\in E\), we have \(M(s)=W\mbk{0;g'_v}\equiv 0\) where \(g'\) is an arbitrary element in \(D_{\epsilon}f\). 
Consequently, \(M(s)\) is single-valued and Lipschitzian on \(E\).
\end{proof}
%\begin{asump}
%The function \(f\) is upper semicontinuous (o.s.c.) at \(\bar{x}\) relative to \textcolor{red}{later deletes this ?}
%\end{asump}
As a byproduct we have the following lemma about strong metrical regularity. 
\begin{lemma}
Suppose  Assumptions \ref{7as:1} and \ref{7as:2} hold. If \(f_{\mathcal{V}}(v)\) is subdifferentially continuous at 0, then the mapping \(\partial F\) is strongly metrically regular at \(\mbk{0,0}\), where \(F\) is defined in proposition \ref{7:pro:5}, for all \(\bar{g}_v\in P_{\mathcal{V}}\mbk{\partial f\mbk{\bar{x}}} \).
\end{lemma}
\begin{proof}
We can easily check that \(F\) is l.s.c. 
From Proposition \ref{7:prop:6} 0 gives a tilt-stable local minimum of \(F\) and hence \(0\in\partial F(0)\). 
We first show that \(F\) is subdifferentially continuous at 0 for 0. 
By definition we need to show \(F\mbk{v^k}\rightarrow F(0)\) for all \(v^k\rightarrow 0\) and all \(s^k\rightarrow 0\) with \(s^k\in\partial F\mbk{v^k}\). 
When \(v^k\) is small enough, \(F\mbk{v^k}=f\mbk{\bar{x}+v^k}-\mang{\bar{g}_v,v^k}\) and by Proposition \ref{7pro:5}, \(\partial F\mbk{v^k}=\partial f_{\mathcal{V}}\mbk{v^k}-\bar{g}_v\). 
Hence each sequence \(s^k\) corresponds to a sequence \(p^k\rightarrow \bar{g}_v\) with \(p^k\in\partial f_{\mathcal{V}}\mbk{v^k}\). As \(\partial f_{\mathcal{V}}(0)=P_{\mathcal{V}}\mbk{\partial f\mbk{\bar{x}}}\) from Lemma \ref{7:lem:u} and \(\bar{g}_v\in P_{\mathcal{V}}\mbk{\partial f\mbk{\bar{x}}}\), we have \(\bar{g}_v\in \partial f_{\mathcal{V}}(0)\).
From the subdifferential continuity of \(f_{\mathcal{V}}\) at 0 we have \(f\mbk{\bar{x}+v^k}\rightarrow f\mbk{\bar{x}}\) and therefore \(F\mbk{v^k}\rightarrow f\mbk{\bar{x}}=F(0)\) for all \(\bar{g}_v\).
From Propositions \ref{7:pro:5} and \ref{7:prop:6}, \(F\) is prox-regular at \(0\) for  0 and 0 gives a tilt-stable local minimum of \(F\). 
We can apply Proposition 3.1 of \cite{drusvyatskiy2013tilt} to conclude that \(\partial F\) is strongly metrically regular at \(\mbk{0,0}\).
\end{proof}
Next, we will show that the U-Lagrangian is prox-regular at \(0\). First we give two basic notions in nonsmooth analysis.  
\begin{definition}[monotonicity] A mapping \(T \colon \mathbb{R}^n \rightrightarrows \mathbb{R}^n\) is called monotone
if it has the property that
\[\mang{v_1 - v_0, x_1 -x_0}\geq 0\text{ whenever }v_0 \in T(x_0 ),\ v_1 \in T(x_1 ).\]
\end{definition}
\begin{definition}
For a function \(f\colon\mathbb{R}^n\rightarrow\bar{\mathbb{R}} \), an \(\epsilon>0\) and \(\bar{v}\in\partial f(\bar{x})\), the \(f\)-attentive \(\epsilon\)-localization of \(\partial f\) around \((\bar{x},\bar{v})\) is the mapping \(T\colon\mathbb{R}^n\rightrightarrows \mathbb{R}^n\) defined by 
\begin{equation}
T(x)=\begin{cases}
\mcb{v\in\partial f(x)\colon\normt{v-\bar{v}}<\epsilon} & \text{if }\normt{x-\bar{x}}<\epsilon\text{ and }|f(x)-f(\bar{x})|<\epsilon,\\
 \emptyset & \text{otherwise}.
\end{cases}
\end{equation}
\end{definition}
\begin{proposition}[13.36 of \cite{Rockafellar1998}]
Suppose
\(f\colon\mathbb{R}^n\rightarrow\bar{\mathbb{R}} \) is finite and locally l.s.c. at \(\bar{x}\), and let \(\bar{v}\in\partial f(\bar{x})\) be a proximal
subgradient. Then the following conditions are equivalent:

(a) \(f\) is prox-regular at \(\bar{x}\) for \(\bar{v}\);

(b) \(\partial f\) has an \(f\)-attentive \(\epsilon\)-localization \(T\) around \((\bar{x}, \bar{v})\) for some \(\epsilon\) such that \(T + \rho I\) is
monotone for some \(\rho\in\mathbb{R}_{+}\).
\end{proposition}
\begin{remark}\label{7rem:1}
A careful examination of 13.36 in \cite{Rockafellar1998}] shows that if \(f\) is prox-regular at \(\bar{x}\) for \(\bar{v}\) with respect to some \(\epsilon \) and \(\rho\) then \(\partial f\) has an \(f\)-attentive \(\epsilon\)-localization \(T\) around \((\bar{x}, \bar{v})\) for the same \(\epsilon\) such that \(T + \rho I\) is
monotone for the same \(\rho\).
\end{remark}
\begin{asump}\label{7as:3}
 Given \(r=\sqrt{\bar{\epsilon}-\epsilon}\) and 
\begin{equation}
\Theta\mathrel{\mathop:}=\{u\in \textrm{int } B_{\mathcal{U}}(0,r)\colon |L_{\epsilon}(u;\bar{g}_v)-f(\bar{x})|<r\},
\end{equation} there exists a constant \(c\in\mathbb{R}_+\) such that for any \(u_i\in\Theta\), \(i\in\{1,2\}\) and any \(v_i\in W(u_i;\bar{g}_v)\),
\begin{equation}\label{7:1eq:metric}
\normt{v_1-v_2}\leq c\normt{u_1-u_2}.
\end{equation}
\end{asump}
\begin{theorem}\label{7th:4}
If Assumptions \ref{7as:1}-\ref{7as:3} hold then $L_{\epsilon}$ is prox-regular at 0 for \(\bar{g}_u\).
\end{theorem}
\begin{proof}
As \(\bar{g}\in D_{\epsilon}f\), from the definition of \(D_{\epsilon}f\) in Definition \ref{7def:1} we have \(\bar{g}\in\partial f(\bar{x})\) From Theorem \ref{7th:2}(ii) we have 
\(h(u,v)\) is prox-regular at \((0,0)\) for \((\bar{g}_u,0)\) with respect to \(\bar{\epsilon}\) and \(\rho\). By Remark \ref{7rem:1} \(\partial h\) has an \(h\)-attentive \(\bar{\epsilon}\)-localization \(T\) around \(\mbk{\mbk{0,0},\mbk{\bar{g}_u,0}}\) such that \(T+\rho I\) is monotone. By definition we have 
\begin{equation}
T(u,v)=\begin{cases}
\mcb{\mbk{w,z}\in \partial h(u,v)\colon \normt{\mbk{w,z}-\mbk{\bar{g}_u,0}}<\bar{\epsilon}} & \text{if }\normt{\mbk{u,v}}<\bar{\epsilon}\text{ and } |h(u,v)-h(0,0)|<\bar{\epsilon}, \\
\emptyset & \text{otherwise.}
\end{cases}
\end{equation}
The monotonicity of \(T+\rho I\) means
\begin{align}\label{7:1eq:mono}
\mang{\mbk{w_1,z_1}-\mbk{w_0,z_0},\mbk{u_1,v_1}-\mbk{u_0,v_0}}+\rho\normt{\mbk{u_1,v_1}-\mbk{u_0,v_0}}^2\geq 0\\\text{ whenever }\mbk{w_0,z_0}\in T\mbk{u_0,v_0},\mbk{w_1,z_1}\in T\mbk{u_1,v_1}.
\end{align}
Consider the \(L_{\epsilon}\)-attentive \(r\)-localization \(S\) of \(\partial L_{\epsilon}\) around \(\mbk{0,\bar{g}_u}\),
\begin{equation}
S(u)=\begin{cases}
\mcb{s\in\partial L_{\epsilon}\mbk{u;\bar{g}_v}\colon\normt{s-\bar{g}_u}<r} &\text{if }\normt{u}<r\text{ and }|L_{\epsilon}\mbk{u;\bar{g}_v}-L_{\epsilon}\mbk{0;\bar{g}_u}|<r,\\
\emptyset & \text{otherwise}.
\end{cases}
\end{equation}
To show $L_{\epsilon}$ is prox-regular at 0 for \(\bar{g}_u\), it suffices to show that there exists \(\hat{\rho}\in\mathbb{R}_+\) such that \(S+\hat{\rho}I\) is monotone. 
Consider any \(s_0\in S\mbk{u_0}\) and \(s_1\in S\mbk{u_1}\). 
As \(S\mbk{u_i}\not=\emptyset\) we have \(\normt{u_i}<r\) and \(|L_{\epsilon}\mbk{u_i;\bar{g}_u}-L_{\epsilon}\mbk{0;\bar{g}_u}|<r\) for \(i\in\mcb{0,1}\). 
Additionally, \(\normt{s_i-\bar{g}_u}<r\) and \(s_i\in\partial L_{\epsilon}\mbk{u_i;\bar{g}_v}\). 
By Theorem \ref{7th:2}(iv), there exist \(\hat{v}_i\in W\mbk{u_i;\bar{g}_v}\), \(i\in\mcb{0,1}\) such that \(\mbk{s_i,0}\in\partial h\mbk{u_i,\hat{v}_i}\). 
Next we show \(\mbk{s_i,0}\in\partial T\mbk{u_i,\hat{v}_i}\). 
First, \(\normt{\mbk{s_i,0}-\mbk{\bar{g}_u,0}}=\normt{s_i-\bar{g}_u}<r<\bar{\epsilon}\). 
Second, \(\normt{(u_i,\hat{v}_i)}^2=\normt{u_i}^2+\normt{\hat{v}_i}^2<r^2+\epsilon^2=\bar{\epsilon}^2\).
Third, \(|h\mbk{u_i,\hat{v}_i}-h(0,0)|=|L_{\epsilon}\mbk{u_i;\bar{g}_u}-L_{\epsilon}\mbk{0;\bar{g}_u}|<r<\bar{\epsilon}\).
Consequently, \(\mbk{s_i,0}\) and \(\mbk{u_i,\hat{v}_i}\) satisfy \eqref{7:1eq:mono}, i.e.
\begin{align}
\mang{s_1-s_0,u_1-u_0}+\rho\normt{u_1-u_0}^2+\rho\normt{\hat{v}_1-\hat{v}_0}^2\geq 0.
\end{align}
Combining \eqref{7:1eq:metric} in Assumption \ref{7as:3} we get \(S+\rho\mbk{1+c^2}I\) is monotone.
\end{proof}
\section{UV decomposition}
In Definition \ref{7def:1} we have defined the subspace to be any subspace such that the set \(D_\epsilon f\) is nonempty. Under this definition we have showed Theorem \ref{7th:3}. However, this definition can be too generic. To get nicer properties of the \(U\)-Lagrangian such as differentiability we follow the definition in \cite{lemarechal2000U-Lagrangian}, where the \(U\)-Lagrangian of a convex function was defined. 
\begin{definition}\label{7def:9}
Given a proper, l.s.c. function \(f\)  and a point \(\bar{x}\), the \(UV\)-decomposition of \(\mathbb{R}^n\) at \(\bar{x}\) is defined by
\begin{equation}
\mathcal{V}(\bar{x})=\spn(\partial f(\bar{x})-\tilde{g} ),\quad \mathcal{U}(\bar{x})=\mathcal{V}\mbk{\bar{
x}}^\bot
\end{equation}
where \(\tilde{g}\) is an arbitrary subgradient in \(\partial f(\bar{x})\).
\end{definition}\label{7def:9}
From now on we replace the subspaces defined in Definition \ref{7def:1} by \(\mathcal{V}\mbk{\bar{
x}} \) and \(\mathcal{U}\mbk{\bar{x}}\) in Definition \ref{7def:9}. To simplify notation we denote \(\mathcal{V}=\mathcal{V}\mbk{\bar{
x}}\) and \(\mathcal{U}=\mathcal{U}\mbk{\bar{
x}}\). Under this definition the set \(D_{\epsilon}f\) becomes the \(\epsilon\)-relative interior of \(\partial f(\bar{x})\), denoted by \(\epri \partial f(\bar{x})\).
\begin{proposition}\label{7prop:7}
Let \(f\) satisfy Assumption \ref{7as:1}. 
Denote \\\(U'=\mcb{w\in\mathbb{R}^n  \colon df(\bar{x})(-w)=-df\mbk{\bar{x},w }}\).

(i)For all \(g^{\circ}\in \epri \partial f\mbk{\bar{x}}\)
\begin{equation}\label{7eq:utem}
\mcb{w\in\mathbb{R}^n\colon \mang{g-g^{\circ},w}=0\text{ for all } g\in\partial f\mbk{\bar{x}}}=N_{\partial f\mbk{\bar{x}}}\mbk{g^{\circ}}
\end{equation}
and \(N_{\partial f\mbk{\bar{x}}}\mbk{g^{\circ}_1}=N_{\partial f\mbk{\bar{x}}}\mbk{g^{\circ}_2}\) for any \(g^{\circ}_1,\ g^{\circ}_2\in \epri \partial f\mbk{\bar{x}}\);

(ii) \(\mathcal{U}=U'=N_{\partial f\mbk{\bar{x}}}\mbk{g^{\circ}} \).
\end{proposition}
\begin{proof}
(i) Take \(g^{\circ}\in\epri \partial f\mbk{\bar{x}}\) and set \(Y=N_{\partial f\mbk{\bar{x}}}\mbk{g^{\circ}}\). 
From Assumption \ref{7as:1} and Proposition \ref{7prop:3} we have \(f\) is subdifferentially regular at \(\bar{x} \). 
Thus  \(\partial f\mbk{\bar{x}}\) is convex and \(Y=\mcb{z\in\mathbb{R}^n
\colon\mang{z,g-g^{\circ}}\leq 0\text{ for all }g\in \partial f\mbk{\bar{x}} }\). 
Thus \(Y\) contains the left-hand side in \ref{7eq:utem}; we only need to show the converse inclusion. Let \(w\in Y\) and \(g\in \partial f\mbk{\bar{x}}\); 
it suffices to prove \(\mang{g-g^{\circ},w}\geq 0 \).
If \(g-g^{\circ}\not=0\) then \(v\mathrel{\mathop:}=-\frac{g-g^{\circ}}{\normt{g-g^{\circ}}}\in\mathcal{V} \) 
and \(\eta v\in B_{\mathcal{V}}(0,\epsilon)\) for all \(\eta\in]0,\epsilon]\). As \(g^{\circ}\in \epri \partial f\mbk{\bar{x}}\) we have \(g^{\circ}+\eta v\in \partial f\mbk{\bar{x}}\). 
The fact that \(w\in Y\) implies that \[
0\geq\mang{ g^{\circ}+\eta v-g^{\circ},w}=-\frac{\eta}{\normt{g-g^{\circ}}}\mang{g-g^{\circ},w}
\]
and therefore \(\mang{g-g^{\circ},w}\geq 0\).
For any \(g^{\circ}_1,\ g^{\circ}_2\in \epri \partial f\mbk{\bar{x}}\subset\partial f\mbk{\bar{x}} \), we have 
\[N_{\partial f\mbk{\bar{x}}}\mbk{g^{\circ}_1}=\mcb{w\in\mathbb{R}^n\colon\mang{g,w}=\mang{g^{\circ}_1,w}=\mang{g^{\circ}_2,w}\text{ for all }g\in \partial f\mbk{\bar{x}}}=N_{\partial f\mbk{\bar{x}}}\mbk{g^{\circ}_2}. \]

(ii) From the regularity of \(f\) we also have \(df\mbk{ \bar{x} }(w)=\sup\mcb{\mang{g,w}\colon g\in\partial f\mbk{\bar{x}}} \) for all \(w\in \mathbb{R}^n \) and 
\begin{equation}\label{7eq:upri}
U'=\mcb{w\in\mathbb{R}^n\colon \sup_{g\in\partial f\mbk{\bar{x}}}\mang{w,g} =\inf_{g\in\partial f\mbk{\bar{x}}}\mang{w,g}}.
\end{equation}
This means for all \(w\in U'\) we have \(\mang{g'-g'',w}=0\) for all \(g',\ g''\in\partial f\mbk{\bar{x}}\). 
From (i) we see \(U'=N_{\partial f\mbk{\bar{x}}}\mbk{g^{\circ}}\).
Now we show \(U'=\mathcal{U}\). Let \(w\in \mathcal{U}\) then \(\mang{w,z}=0\) for any \(z\in\mathcal{V}\). Specifically, \(\mang{w,g-\tilde{g}}=0\) for any \(g\in\partial f\mbk{\bar{x}}\) as \(g-\tilde{g}\in\mathcal{V}\).  Consequently, \(\mang{g'-g'',w}=0\) for all \(g',\ g''\in\partial f\mbk{\bar{x}}\) and \(w\in U'\), meaning \(\mathcal{U}\subset U'\).
Suppose \(w\in U'\) and \(v=\sum_j\lambda_j\mbk{g_j-\tilde{g}}\in\mathcal{V}\) with \(g_j\in \partial f\mbk{\bar{x}}\), then we have from \eqref{7eq:upri} and the fact that \(g^{\circ}\in\epri \partial f\mbk{\bar{x}}\subset\partial f\mbk{\bar{x}}\)
\begin{equation}
\mang{v,w}=\sum_j\lambda_j\mbk{\mang{g_j,w}-\mang{\tilde{g},w}}=0
\end{equation}
and hence \(w\in\mathcal{V}^{\bot}=\mathcal{U}\). This finishes the proof.
\end{proof}
\begin{corollary}\label{7cor:1}
Suppose that Assumption \ref{7as:1} holds. For all \(g\in\partial f\mbk{\bar{x}}\) we have \(g_u=\bar{g}_u\), where \(\bar{g}\in D_{\epsilon}f= \epri \partial f\mbk{\bar{x}}\) is introduced in Definition \ref{7def:1}.
\end{corollary}
\begin{proof}
By Proposition \ref{7prop:7} we see \(\mathcal{U}=\mcb{w\in\mathbb{R}^n\colon \mang{g-g^{\circ},w}=0\text{ for all } g\in\partial f\mbk{\bar{x}}}\). Consequently, for all \(u\in\mathcal{U}\) and all \(g\in \partial f\mbk{\bar{x}}\) we have \(\mang{g,u}=\mang{g^{\circ},u}\). As \(g^{\circ}\) and \(\bar{g}\) are all in \(\partial f\mbk{\bar{x}}\) it follows that \(\mang{g,u}=\mang{\bar{g},u}\).  By the \(UV\)-decomposition, we get \(\mang{g_u,u}=\mang{\bar{g}_u,u}\) and \(\mang{g_u-\bar{g}_u,u}=0\) for all \(u\in \mathcal{U}\). Therefore \(g_u=\bar{g}_u\). 
\end{proof}
Now we show the smoothness property of the U-Lagrangian. 
\begin{proposition}\label{7pro:8}
Suppose Assumptions \ref{7as:1} and \ref{7as:2} hold. Then \(\partial L_{\epsilon}(0;\bar{g}_v)=\mcb{\bar{g}_u}\).
\end{proposition}
\begin{proof}
By Assumption \ref{7as:1} and Theorem \ref{7th:2}(i) we can apply 10.13 of \cite{Rockafellar1998} to deduce 
\(\partial L_{\epsilon}(0;\bar{g}_v)\subset\cup_{\hat{v}\in W\mbk{0;\bar{g}_v}}\mcb{w\colon\mbk{w,0}\in\partial h\mbk{0,\hat{v }}}\). 
From Theorem \ref{7th:3}(i) we have 
\(\partial L_{\epsilon}(0;\bar{g}_v)\subset\mcb{w\colon\mbk{w,0}\in\partial h\mbk{0,0}}\).
Applying Theorem \ref{7th:2}(ii) we get 
\(\partial L_{\epsilon}(0;\bar{g}_v)\subset\mcb{g_u\colon g_v=\bar{g}_v,\ g\in\partial f\mbk{\bar{x}}}\). Combining Corollary \ref{7cor:1} we have \(\partial L_{\epsilon}(0;\bar{g}_v)\subset\mcb{g_u\colon g_v=\bar{g}_v,\ g_u=\bar{g}_u}=\mcb{\bar{g}_u}\). On the other hand, Theorem \ref{7th:3}(ii) gives \(\bar{g}_u\in\partial_p L_{\epsilon}\mbk{0;\bar{g}_v}\subset \partial L_{\epsilon}\mbk{0;\bar{g}_v}\). Consequently, \(\partial L_{\epsilon}(0;\bar{g}_v)=\mcb{\bar{g}_u}\) holds.
\end{proof}
\begin{corollary}\label{7cor:2}
If Assumptions \ref{7as:1}-\ref{7as:3} hold then $L_{\epsilon}$ is prox-regular at 0.
\end{corollary}
\begin{proof}
From Proposition \ref{7pro:8} we see \(\partial L_{\epsilon}(0;\bar{g}_v)\) is a singleton. The result follows immediately from Theorem \ref{7th:4}.
\end{proof}
\begin{proposition}\label{7pro:9}
Suppose Assumption \ref{7as:1} holds. Define \(f_{\mathcal{U}}\colon \mathcal{U}\mapsto \bar{R}\) as \(f_{\mathcal{U}}(u)=f(\bar{x}+u)\). Then \(f_{\mathcal{U}}\) is strictly differentiable at 0.
\end{proposition}
\begin{proof}
The function \(f_{\mathcal{U}}\) is l.s.c. at 0 because \\
\(\liminf_{u\rightarrow 0}f_{\mathcal{U}}(u)=\lim_{\delta\searrow 0}\left[\inf_{u\in B_{\mathcal{U}}(0,\delta)}f_{\mathcal{U}}(u)\right]\geq \lim_{\delta\searrow 0}\left[\inf_{u\in B(0,\delta)}f(\bar{x}+u)\right]=\liminf_{x\rightarrow\bar{x}}f(x)\geq f(\bar{x})=f_{\mathcal{U}}(0)\), where the last inequality holds because \(f\) is l.s.c. at \(\bar{x}\) by Assumption \ref{7as:1}. 
According to 9.18 (g) of \cite{Rockafellar1998} it suffices to show \(\hat{d}f_{\mathcal{U}}(0)(-p)=-\hat{d}f_{\mathcal{U}}(0)(p)\) for all \(p\in \mathcal{U}\). 
From Proposition \ref{7prop:3}(ii) and Assumption \ref{7as:1}, \(f\) is subdifferentially regular at \(\bar{x}\).
On the other hand, \(f_{\mathcal{U}}(u)=f(F(u))\) where \(F(u)=\bar{x}+u\) and \(\nabla F(0)\) is an identity matrix. This implies that the only vector \(y\) with \(\nabla F(0)*y=0 \) is \(y=0\). Consequently, \(f_{\mathcal{U}}\) is  subdifferentially regular at 0 and \(df_{\mathcal{U}}(0)(p)=df(\bar{x})(p)\) for all \(p\in \mathcal{U}\) (see 10.6 of \cite{Rockafellar1998}). 
Thus the equivalence of \(\mathcal{U}\) and \(U'\) in Proposition \ref{7prop:7} implies  
\(\hat{d}f_{\mathcal{U}}(0)(-p)=df_{\mathcal{U}}(0)(-p)=df(\bar{x})(-p)=-df(\bar{x})(p)=-df_{\mathcal{U}}(0)(p)=-\hat{d}f_{\mathcal{U}}(0)(p)\).
\end{proof}
\begin{lemma}\label{7lem:1}
Under Assumptions \ref{7as:1} and \ref{7as:2}, \(L_{\epsilon}\) is strictly continuous at \(0\).
\end{lemma}
\begin{proof}
The definition of \(L_{\epsilon}\) suggests \(L_{\epsilon}(u;\bar{g}_v)\leq f(\bar{x}+u)\) for all \(u\in\mathcal{U}\) and therefore \(\limsup_{u\rightarrow 0}L_{\epsilon}(u;\bar{g}_v)\leq \limsup_{u\rightarrow 0}f(\bar{x}+u)\). 
Proposition \ref{7pro:9} entails the continuity of \(f_{\mathcal{U}}\) at 0 which gives \(\limsup_{u\rightarrow 0}f_{\mathcal{U}}(u)=\limsup_{u\rightarrow 0}f(\bar{x}+u)=f_{\mathcal{U}}(0)=f(\bar{x})\). Consequently, we have \(\limsup_{u\rightarrow 0}L_{\epsilon}(u;\bar{g}_v)\leq f(\bar{x})=L_{\epsilon}(0;\bar{g}_v)\). 
On the other hand, \ref{7th:3}(iii) reveals that \(\liminf_{u\rightarrow 0}L_{\epsilon}\mbk{u;\bar{g}_v}\geq f(\bar{x})\) and therefore \(L_{\epsilon}\) is continuous at 0. 
From the proof hitherto we see that 
\[f(\bar{x})+\mang{\bar{g}_u,u}-\frac{\rho}{2}\normt{u}^2\leq L_{\epsilon}\mbk{u;\bar{g}_v}\leq f(\bar{x}+u)\ \forall \,  u\in B_{\mathcal{U}}(0,r), \text{ where }r=\sqrt{\bar{\epsilon}-\epsilon}.\] 
By Proposition \ref{7pro:9} \(f(\bar{x})+u\) is strictly continuous at 0 and hence \(L_{\epsilon}\) is bounded below and above by functions that are strictly continuous at 0. It follows that \(L_{\epsilon}\) must be strictly continuous at 0, too. 
\end{proof}
\begin{theorem}\label{7the:4}
Under Assumptions \ref{7as:1} and \ref{7as:2}, \(L_{\epsilon}\) is strictly differentiable at \(0\) with \(\nabla L_{\epsilon}\mbk{0;\bar{g}_v}=\bar{g}_u\).
\end{theorem}
\begin{proof}
By Proposition \ref{7pro:8}, \(L_{\epsilon}\) has only one subgradient at 0. Applying 9.18 of \cite{Rockafellar1998}, we see it suffices to show \(L_{\epsilon}\) is strictly continuous at 0, which is shown in Lemma \ref{7lem:1}.
\end{proof}
\section{Fast track}
In this section, we first extends fast track to the new U-Lagrangian and then we explore its equivalence to partial smoothness. 
\begin{definition}[\(\mathcal{C}^k\) fast track]\label{7def:10}
Let \(\bar{x}\) be a local minimizer of the function \(f\colon \mathbb{R}^n\mapsto \bar{\mathbb{R}}\).   
We say that \(\mcb{\bar{x}+u+v(u)\colon u\in B_{\mathcal{U}}\mbk{0,\delta}}\) is a \(\mathcal{C}^k\) fast track leading to a local minimizer of \(f\) if for all \(u\) small enough\\
(i) \(v\colon \mathcal{U}\mapsto \mathcal{V}\) is a \(\mathcal{C}^k\) function satisfying \(v(u)\in \bigcap\limits_{\bar{g}\in\epri \partial f\mbk{\bar{x}}} W\mbk{u;\bar{g}_v}\); and\\
(ii) there exists \(\hat{g}\in \epri\partial f\mbk{\bar{x}} \) such that \(L_{\epsilon}\mbk{u;\hat{g}_v}\) is a \(\mathcal{C}^k\) function. 
\end{definition}
\begin{proposition}\label{7prop:11}
In Definition \ref{7def:10}, the condition (ii) can be replaced by the following statement,\\
(ii\(^*\)) \(L_{\epsilon}\mbk{u;\bar{g}_v}\) is a \(\mathcal{C}^1\) function for all \(\bar{g}\in \epri\partial f\mbk{\bar{x}}\).
\end{proposition}
\begin{proof}
Suppose (ii) holds. Let \(v\mbk{u;\hat{g}_v}\) be an arbitrary element in \(W\mbk{u;\gbarv}\). Then \(L_{\epsilon}\mbk{u;\hat{g}_v}=f\mbk{\bar{x}+u+v\mbk{u;\hat{g}_v}}-\mang{\hat{g}_v,v\mbk{u;\hat{g}_v}}\) and \(f\mbk{\bar{x}+u+v\mbk{u;\hat{g}_v}}-\mang{\hat{g}_v,v\mbk{u;\hat{g}_v}}\) is a \(\mathcal{C}^1\) function of \(u\). 
Taking \(v\mbk{u;\hat{g}_v}\) as the particular \(v(u)\) in condition (i) of Definition \ref{7def:10} 
we have 
\(L_{\epsilon}\mbk{u;\hat{g}_v}=f\mbk{\bar{x}+u+v(u)}-\mang{\hat{g}_v,v(u)}=f\mbk{\bar{x}+u+v(u)}-\mang{\bar{g}_v,v(u)}+\mang{\gbarv-\hat{g}_v,v(u)}=L_{\epsilon}\mbk{u;\bar{g}_v}+\mang{\gbarv-\hat{g}_v,v(u)}\) for all \(\bar{g}\in\epri\partial f\mbk{\bar{x}}\). 
Therefore \(L_{\epsilon}\mbk{u;\bar{g}_v}=L_{\epsilon}\mbk{u;\hat{g}_v}-\mang{\gbarv-\hat{g}_v,v(u)}\) is also a \(\mathcal{C}^1\) function.
\end{proof}
\begin{definition}[\(\mathcal{C}^k\)-manifold]
We say that a set \(\mathcal{M}\subset\mathbb{R}^n\) is a \(\mathcal{C}^k\)-smooth manifold of codimension \(m\) around \(\bar{y}\in\mathcal{M}\) if there is an open set \(Q\subset\mathbb{R}^n\) such that 
\[\mathcal{M}\cap Q=\mcb{y\in Q\colon\phi_i(y)=0,\ i=1,\cdots
,m},\]
where \(\phi_i\) are \(\mathcal{C}^k\) functions with \(\nabla\phi_i\mbk{\bar{y}}\) linearly independent.  
\end{definition}
As this paper will only involve manifolds that has only one chart, a \(\mathcal{C}^1\)-manifold 
 with codimension \(m \) 
can also be defined as the following:\\
\(\mathcal{M}=\mcb{G(z)\colon\  z\in Q}\) where \(Q\) is an open subset of \(\mathbb{R}^n\) and \(G\colon Q\mapsto \mathbb{R}^n\) has surjective derivative throughout \(Q\). 
In this case, it is known that \(T_{\mathcal{M}}\mbk{\bar{y}}=\func{Im}\mbk{\nabla G\mbk{\bar{y}}}\).
\begin{definition}[\(\mathcal{C}^k\)-partly smooth function]\label{7def:13}
Let \(\mathcal{M}\) be a \(\mathcal{C}^k\)-smooth manifold around \(\bar{x}\). We say the function \(h\colon\mathbb{R}^n\rightarrow\bar{\mathbb{R}} \) is \(\mathcal{C}^k\)-partly smooth at \(\bar{x}\) relative to \(\mathcal{M}\) if the following four properties hold:

\indent (i) there is an open neighborhood \(\mathcal{N}\mbk{\bar{x}}\subset\mathbb{R}^n\) such that some \(\mathcal{C}^k\)-smooth function \(g\colon\mathcal{N}\mbk{\bar{x}}\rightarrow\mathbb{R}\) agrees with \(h\) on \(\mathcal{M}\cap \mathcal{N}\mbk{\bar{x}}\);

\indent (ii) at every point close to \(\bar{x}\) in \(\mathcal{M}\), \(h\) is subdifferentially regular and has a subgradient;

\indent (iii) \(N_{\mathcal{M}}\mbk{\bar{x}}=\mathcal{V}\mbk{\bar{x}}\), where \(\mathcal{V}\mbk{\bar{x}}\) is defined in Definition \ref{7def:9};

\indent (iv) \(\partial h\) is continuous at \(\bar{x}\) relative to \(\mathcal{M}\).
\end{definition}
The following proposition is part of Theorem 6.1 of \cite{Lewis2002active}. 
\begin{proposition}\label{7pro:12}
Suppose the function \(h\colon\mathbb{R}^n\rightarrow\bar{\mathbb{R}} \) is \(\mathcal{C}^2\)-partly smooth at the point \(\bar{x}\) relative to the set \(\mathcal{M}\subset\mathbb{R}^n\). 
Define subspaces \(\mathcal{U}=T_{\mathcal{M}}\mbk{\bar{x}}\) and \(\mathcal{V}=N_{\mathcal{M}}\mbk{\bar{x}}\).
Then there exists a function \(v\colon \mathcal{U}\rightarrow\mathcal{V}\) with the following properties:

\indent (i) the function \(v\) is of class \(\mathcal{C}^2\) near the origin;

\indent (ii) for small vectors \(u\in\mathcal{U}\) and \(w\in\mathcal{V}\), \(\bar{x}+u+w\in\mathcal{M}\Leftrightarrow w=v(u)\);

\indent (iii) \(v(u)=\mathcal{O}\mbk{\normt{u}^2}\) for small \(u\in\mathcal{U}\).

\noindent Fix any vector \(y\in \func{ri}\partial h\mbk{\bar{x}}\). Then for any small vector \(u\in\mathcal{U}\), the function
\begin{equation}
w\in\mathcal{V}\mapsto h\mbk{\bar{x}+u+w}-\mang{y,\bar{x}+u+w}
\end{equation}
has a local minimizer at the point \(v(u)\).
\end{proposition}
\begin{theorem}
Suppose \(\bar{x}\) is a local minimizer of the function \(f\colon\mathbb{R}^n\rightarrow\bar{\mathbb{R}} \). 
If \(f\) is \(\mathcal{C}^2\)-partly smooth at \(\bar{x}\) relative to the manifold \(\mathcal{M}\), then \(\mathcal{M}\) contains a \(\mathcal{C}^2\) fast track leading to \(\bar{x}\).
\end{theorem}
\begin{proof}
From property (iii) in Definition \ref{7def:13} we see that the subspaces \(\mathcal{U}\) and \(\mathcal{V}\) in Definition \ref{7def:10} and that in Proposition \ref{7pro:12} are the same. 
Applying Proposition \ref{7pro:12}, we get that fix any \(g\in \func{ri}\partial f\mbk{\bar{x}}\), there exists a number dependent on \(g\) and denoted by \(\delta(g)\) such that 
\begin{equation}\label{7eq:the5n1}
\fa u\in B_{\mathcal{U}}\mbk{0,\delta(g)},\ \exists\,\epsilon(g,u),\ \text{such that }v(u)\in \argmin\limits_{v\in B_{\mathcal{V}}\mbk{0,\epsilon(g,u)}}\mcb{f\mbk{\bar{x}+u+v}-\mang{g,\bar{x}+u+v}}, 
\end{equation}
where \(\epsilon(g,u)\) is a number dependent on \(g\) and \(u\).
In \eqref{7eq:the5n1} we can choose \(\epsilon\) sufficiently small so that it does not depend on \(u\). 
As \(g\) can be any element in \(\func{ri}\partial f\mbk{\bar{x}}\)  we can choose \(\delta\) and \(\epsilon\) sufficiently small so that they do not depend on \(g\). 
That is 
\begin{equation}
\exists\,\delta>0\text{ and }\epsilon>0\text{ such that } v(u)\in \argmin\limits_{v\in B_{\mathcal{V}}(0,\epsilon)}\mcb{f\mbk{\bar{x}+u+v}-\mang{g,\bar{x}+u+v}}, \fa u\in B_{\mathcal{U}}(0,\delta).
\end{equation}
As \(\epri \partial f\mbk{\bar{x}}\subset \func{ri} \partial f\mbk{\bar{x}}\), we can take any \(\bar{g}\in \epri \partial f\mbk{\bar{x}}\) and get 
\begin{gather}
v(u)\in \argmin\limits_{v\in B_{\mathcal{V}}(0,\epsilon)}\mcb{f\mbk{\bar{x}+u+v}-\mang{\bar{g},\bar{x}+u+v}}\\
= \argmin\limits_{v\in B_{\mathcal{V}}(0,\epsilon)}\mcb{f\mbk{\bar{x}+u+v}-\mang{\bar{g}_v,v}}\\
=W(u;\bar{g}_v),
\end{gather}
where \(W(u;\bar{g}_v)\) is defined in Definition \ref{def:1}. 
Additionally we have \(v(u)\in \bigcap\limits_{\bar{g}\in\epri \partial f\mbk{\bar{x}}} W\mbk{u;\bar{g}_v}\).
From properties (ii) and (iii) in Proposition \ref{7pro:12} we can choose \(\delta\) sufficiently small such that \(\mathcal{M}'\mathrel{\mathop:}=\mcb{\bar{x}+u+v(u)\colon u\in B_{\mathcal{U}}\mbk{0,\delta}}\subset\mathcal{M}\). 
From property (i) in Definition \ref{7def:13} we can choose \(\delta\) sufficiently small such that \(\mathcal{M}'\) is contained in some open neighborhood \(\mathcal{N}\mbk{\bar{x}}\subset\mathbb{R}^n\). 
Consequently, \(f|_{\mathcal{M}'}\) is of class \(\mathcal{C}^2\). 
From property (i) in Proposition \ref{7pro:12} we can shrink \(\delta\) if necessary to guarantee \(v(u)\) is of class \(\mathcal{C}^2\) on \(B_{\mathcal{U}}\mbk{0,\delta}\). 
By the definition of \(W\mbk{u;\bar{g}_v}\) we have on \(B_{\mathcal{U}}\mbk{0,\delta}\) 
\[L_{\epsilon}\mbk{u;\bar{g}_v}=f\mbk{\bar{x}+u+v(u)}-\mang{\bar{g}_v,v(u)}\in\mathcal{C}^2.\]
We have verified Definition \ref{7def:10} and hence \(\mathcal{M}'\) is a \(\mathcal{C}^2\) fast track leading to \(\bar{x}\).
\end{proof}
%\begin{proposition}
%Let \(f\colon \mathbb{R}^n\mapsto \bar{\mathbb{R}}\) satisfy Assumptions \ref{7as:1} and \ref{7as:2} and \(\bar{x}\) be a local minimizer of \(f\). If \(L_{\epsilon}\mbk{u;0}\) is a \(\mathcal{C}^1\) function then for all \(\bar{g}_v\in \epri \partial f\mbk{\bar{x}}\), \(L_{\epsilon}\mbk{u;\bar{g}_v}\) is of class \(\mathcal{C}^1\) too and \(\nabla L_{\epsilon}\mbk{u;0}= \nabla L_{\epsilon}\mbk{u;\bar{g}_v}\).
%\end{proposition}
%\begin{proof}
%As \(L_{\epsilon}\mbk{u;0}=\inf_{v\in B_{\mathcal{V}}(0,\epsilon)}f\mbk{\bar{x}+u+v}\) is a \(\mathcal{C}^1\) function we have \(\mcb{\nabla L_{\epsilon}\mbk{u;0}} =\hat{ \partial} L_{\epsilon}\mbk{u;0}	\). 
%We can express \(L_{\epsilon}\mbk{u;0}\) as a marginal function \(\inf_{v\in\mathbb{R}^n}\mcb{h'\mbk{u,v}}\) where \(h'\mbk{u,v}=f\mbk{\bar{x}+u+v}+\delta_{B_{\mathcal{V}}(0,\epsilon)}(v) \). 
%Applying 10.13 of \cite{Rockafellar1998} yields \(\hat{\partial} L_{\epsilon}\mbk{u;0}\subset \cap_{\bar{v}\in W\mbk{u;0}}\mcb{s\colon (s,0)\in\hat{\partial}h'\mbk{u,\bar{v}}}\subset \cap_{\bar{v}\in W\mbk{u;0}\cap \intrr B_{\mathcal{V}}\mbk{0,\epsilon}}\mcb{s\colon (s,0)\in\hat{\partial}h'\mbk{u,\bar{v}}} \). 
%%If \(\bar{v}\in \intrr B_{\mathcal{V}}\mbk{0,\epsilon}\) then \(\hat{\partial}h'\mbk{u,\bar{v}}=\hat{\partial}f(\bar{x}+u+\bar{v})\). If \(\bar{v}\) is on the boundrary of \(B_{\mathcal{V}}\mbk{0,\epsilon}\) then 
%\end{proof}
Suppose the dimension of \(\mathcal{U}\) is \(m\) and the dimension of \(\mathcal{V}\) is \(p\mathrel{\mathop:}=n-m\).
Suppose that \(\bar{U}\) is a basis matrix for \(\mathcal{U}\) and \(\bar{V}\) is a basis matrix for \(\mathcal{V}\).
We know that if the columns of an \(m'\times n'\) matrix \(A\) are linearly independent then the Moore–Penrose pseudoinverse of \(A\) is \(A^+\mathrel{\mathop:}=\mbk{A^\top A}\inv A^\top\). 
Consequently, every \(x\in\mathbb{R}^n\) can be decomposed into components \(x_u\) and \(x_v\) as follows:
\begin{gather}
\mathbb{R}^n\ni x=\bar{U}x_u+\bar{V}x_v=x_u\oplus x_v\in\mathbb{R}^m\times\mathbb{R}^p,\text{ with}\\
 x_u=\bar{U}^+x \text{ and }x_v=\bar{V}^+x.
\end{gather}

\begin{definition}\label{def:1}
Given $\epsilon>0$,
we take an arbitrary 
\(\bar{g}\in \epri\partial f\mbk{\bar{x}} \)
and define the function $L_{\epsilon}$ as follows:
\begin{equation}\label{7eq:i12}
\mathbb{R}^m\ni u\mapsto L_{\epsilon}(u;\bar{g}_{v})\mathrel{\mathop:}=\inf_{v\in B^\epsilon}\mcb{f(\bar{x}+\bar{U} u+\bar{V} v)-\mang{\bar{g},\bar{V} v}},
\end{equation}
where \(B^\epsilon\mathrel{\mathop:}=\mcb{v\in  {\mathbb{R}^p}\colon \normt{\bar{V}v}\leq \epsilon}\).
Associated with \eqref{7eq:i12} we have the set 
%of $\mathcal{V}$-space minimizers 
\begin{equation}
W(u;\gbarv)\mathrel{\mathop:}=\mcb{v\in\mathbb{R}^p\colon L_{\epsilon}\mbk{u;\gbarv}=  f(\bar{x}+\bar{U} u+\bar{V} v)-\mang{\bar{g},\bar{V} v}}.
\end{equation}
\end{definition}
%Given \(\delta\in\mathbb{R}_+\) define \(B\mathrel{\mathop:}=\intrr B_{\mathcal{U}}\mbk{0,\delta}\).

\begin{lemma}\label{7lem:4}
Suppose Assumptions \ref{7as:1} and \ref{7as:2} hold. Consider a \(U\)-Lagrangian \(L_{\epsilon}\mbk{u;\bar{g}_v}\). 
Let \(v\mbk{u;\gbarv}\) be a function from \(\mathbb{R}^m\) to \(W(u;\gbarv)\subset \mathbb{R}^p\). 
If there exists \(\delta\in\mathbb{R}_+\) such that both \(L_{\epsilon}\mbk{u;\bar{g}_v}\) and \(v\mbk{u;\gbarv}\) are 
of class \(\mathcal{C}^k\) 
on 
\(B^\delta\mathrel{\mathop:}=\mcb{u\in\mathbb{R}^m\colon\normt{\bar{U}u }< \delta }\) %\textcolor{red}{later I need to have a better notation.}
%\(B\mathrel{\mathop:}=\intrr B_{\mathcal{U}}\mbk{0,\delta}\)
, then\\
(i) \(\mathcal{M}\mathrel{\mathop:}=\mcb{\bar{x}+\bar{U}u+\bar{V}v\mbk{u;\gbarv}\colon u\in B^{\delta}}\) is a \(\mathcal{C}^k\)-smooth manifold;\\
(ii) \(\nabla v \mbk{0;\gbarv}=0\), \(v \mbk{u;\gbarv}=o\mbk{\normt{u}}\); and\\
(iii) \begin{equation}\label{7eq:lem2}
T_{\mathcal{M}}\mbk{\bar{x}}=\mathcal{U}.
\end{equation}\\
(iv) The restriction \(f\lvert_\mathcal{M}\) is of class \(\mathcal{C}^k\). 
\end{lemma}
\begin{proof}
Define 
\begin{equation}
\begin{split}
G\colon & \mathbb{R}^m \rightarrow \mathbb{R}^n\\
 & u\mapsto\bar{x}+\bar{U}u+\bar{V}v\mbk{u;\gbarv}.
\end{split}
\end{equation}
We know \(G(0)=0\) because \(v\mbk{0;\gbarv}=0\) from Theorem \ref{7th:3}(i). 
As \(v\mbk{u;\gbarv}\) is of class \(\mathcal{C}^k\), and hence \(G\) is 
of class \(\mathcal{C}^k\)
%continuously differentiable 
on \(B^\delta\). 
It follows that \(\mathcal{M}=\mcb{G(u)\colon  u\in B^\delta }\) is a \(\mathcal{C}^k\)-smooth manifold around \(\bar{x}=G(0)\) provided \(\nabla G(0)\) is injective, where the Jacobian \(\nabla G(0)\) is a \(n\times m\) matrix. 
Take a \(u\in\mathbb{R}^m\) and suppose \(\nabla G(0)u=\bar{U}u+\bar{V}\nabla v\mbk{0;\gbarv}u=0\), then \(\bar{U}u=-\bar{V}\nabla v\mbk{0;\gbarv}u\). 
The left-hand side is an element in \(\mathcal{U}\) and the right-hand side is an element in \(\mathcal{V}\). As the common elements of the two subspaces can only be 0, we get \(u\) must be 0. 
This shows that \(\nabla G(0)\) is injective and therefore the rank of \(\nabla G(0)\) is \(m\).

We know that \(\func{rank}\mbk{\nabla G(0)^\top}+\func{Null}\mbk{\nabla G(0)^\top}=n\). So \(\func{Null}\mbk{\nabla G(0)^\top}=n-m=p\).
Now we show that the kernel of \(\nabla G(0)^\top\) is \(\mathcal{V}\). Take any \(y\in\mathcal{U}\) and set \(\nabla G(0)^\top y=0 \), then \(\bar{U}^\top y+ \mbk{\bar{V}\nabla v\mbk{0;\gbarv}}^\top y=\bar{U}^\top y=0\). As \(\bar{U}\) is an arbitrary basis matrix of \(\mathcal{U}\), we can take an orthogonal one so that \(\mcb{\bar{U}^\top z\colon z\in \mathcal{U}}=\mathbb{R}^m\). 
Therefore \(\bar{U}^\top y=0\) implies \(y\) must be 0. 
Hence the kernel of \(\nabla G(0)^\top\) cannot contain any non zero element in \(\mathcal{U}\) and it must be \(\mathcal{V}\). 
That is, for any \(z\in\mathcal{V}\), \(\nabla G(0)^\top z=0= \bar{U}^\top z+ \mbk{\bar{V}\nabla v\mbk{0;\gbarv}}^\top z=\nabla v\mbk{0;\gbarv}^\top\bar{V}^\top z\).
We can take an orthogonal basis matrix of \(\mathcal{V}\) so that \(\mcb{\bar{V}^\top z\colon z\in \mathcal{V}}=\mathbb{R}^p\).
Consequently, we get \(\nabla v\mbk{0;\gbarv}^\top w=0\) for all \(w\in \mathbb{R}^p\) and thus \(\nabla v\mbk{0;\gbarv}=0\).

As \(v\mbk{u;\gbarv}\) is of class \(\mathcal{C}^k\), it follows that 
\(v\mbk{u;\gbarv}=v\mbk{0;\gbarv}+\nabla v\mbk{0;\gbarv}\mbk{u-0}+o\mbk{\normt{u-0}}\). 
Notice that both \(v\mbk{0;\gbarv}\) and \(\nabla v\mbk{0;\gbarv}\) are 0, and therefore \(v \mbk{u;\gbarv}=o\mbk{\normt{u}}\).
 Next, we show \eqref{7eq:lem2}. 
 Since \(T_{\mathcal{M}}\mbk{\bar{x}}=\func{Image}\mbk{\nabla G\mbk{0}}=\mcb{\bar{U}u+\bar{V}\nabla v\mbk{0;\gbarv}u\colon u\in\mathbb{R}^m }\), it suffices to show \(\nabla v\mbk{0;\gbarv}\) is a zero matrix.

From the definition of \(L_{\epsilon}\mbk{u;\bar{g}_v}\) and \(v\mbk{u;\gbarv}\) we have 
 \begin{equation}
 L_{\epsilon}\mbk{u;\bar{g}_v}=f\mbk{\bar{x}+\bar{U}u+\bar{V}v\mbk{u;\gbarv}}-\mang{\bar{g},\bar{V}v\mbk{u;\gbarv}},\ \forall \,  u\in B^\delta.
\end{equation}
As both \(L_{\epsilon}\mbk{u;\bar{g}_v}\) and \(v\mbk{u;\gbarv}\) are \(\mathcal{C}^k\)-smooth on \(B^\delta\), it follows that \(f\lvert_\mathcal{M}\) is \(\mathcal{C}^k\)-smooth.
% Denote \(h(u)=f\mbk{\bar{x}+\bar{U}u+\bar{V}v\mbk{u;\gbarv}}=f\mbk{G(u)}\).
% Applying Proposition \ref{7pro:5} yields \(\partial L_{\epsilon}\mbk{u;\bar{g}_v}=\partial h(u)-\mbk{\bar{V}\nabla v\mbk{u;\gbarv} }^\top \bar{g}\) for all \(u\in B\). 
% Now to apply the basic chain rule of subdifferential calculus on \(h\) at 0 we verify the following two conditions (see 10.6 of \cite{Rockafellar1998}). 
% The first condition is that the only vector \(z\in\partial^{\infty}f\mbk{G(0)}\) with \(\nabla G(0)\)
\end{proof}
\begin{remark}
Comparing the conditions in Lemma \ref{7lem:4} and Definition \ref{7def:10}, we see the set \(\mathcal{M}\) in Lemma \ref{7lem:4} is almost a fast track except that it depends on the parameter \(\bar{g}_v\) because of the function \(v\mbk{u;\bar{g}_v}\). From now on, we consider a fast track \(\mathcal{M}\mathrel{\mathop:}=\mcb{\bar{x}+\bar{U}u+\bar{V}v\mbk{u}\colon u\in B^\delta}\) and study the conditions under which is can corresponds to a partly smooth function.
\end{remark}
\begin{asump}\label{7as:4}
There exists a real number \(\hat \delta\) such that \(\hat{\partial}f\mbk{x}=\partial f\mbk{x}\) for all \(x\in B\mbk{\bar{x},\hat{\delta}}\cap\mathcal{M}\).
\end{asump}
\begin{asump}\label{7as:5}
There exists a real number \(\tau\in]0,\delta]\) such that for all \(\bar{u}\in B^\tau\mathrel{\mathop:}=\mcb{u\in\mathbb{R}^m\colon \normt{\bar{U}u}\leq \tau}\) we have \(W\mbk{\bar{u};\gbarv}\bigcap\int B^\epsilon\not=\emptyset \) and \(v(\bar{u})\in\int B^\epsilon\) for all \(\bar{u}\in B^\tau\).
\end{asump}
\begin{asump}\label{7:as6}
Suppose the following is true
%\begin{equation}
%\fa g\in \partial f\mbk{\bar{x}}\setminus\epri \partial f\mbk{\bar{x}},\exists\,g^k\rightarrow g\text{ with }g^k\in \epri \partial f\mbk{\bar{x}}.
%\end{equation}
\begin{equation}
P_{\mathcal{V}}\mbk{\partial f\mbk{\bar{x}}\setminus\epri \partial f\mbk{\bar{x}}}\subset \liminf\limits_{x\overset{\mathcal{M}}{\rightarrow}\bar{x}}P_{\mathcal{V}}\mbk{\partial f\mbk{{x}}}.
\end{equation}
\end{asump}
%\begin{lemma}\label{7:lem5}
%\textcolor{red}{later I need to replace this lemma and references to it.} If \(\bar{x}\) is a local minimizer of \(f\) and \(\mathcal{M}\mathrel{\mathop:}=\mcb{\bar{x}+\bar{U}u+\bar{V}v\mbk{u}\colon\bar{U} u\in B}\) is a \(\mathcal{C}^k\) fast track of \(f\), 
%then the restriction \(f|_{\mathcal{M}}\) is of class \(\mathcal{C}^k\).
%\end{lemma}
%\begin{proof}
%By the definitions of \(L_{\epsilon}\mbk{u;\bar{g}_v}\) and \(v(u)\), we have 
%\begin{equation*}
%L_{\epsilon}\mbk{u;\bar{g}_v}=f\mbk{\bar{x}+\bar{U}u+\bar{V}v\mbk{u}}-\mang{\bar{g},\bar{V}v\mbk{u}},\ \forall \,  u\in B^\delta.
%\end{equation*}\textcolor{red}{What is \(B^\delta\)? It is \(B^\delta=\mcb{u\in\mathbb{R}^m\colon \bar{U}u\in B}\).}
%We define \(g(u)\mathrel{\mathop:}=f\mbk{G(u)}\) where \(G(u)\mathrel{\mathop:}=\bar{x}+\bar{U}u+\bar{V}v(u)\). 
%From Definition \ref{7def:10} and Proposition \ref{7prop:11} we see that both \(L_\epsilon\mbk{u;\gbarv}\) and \(v(u)\) are of class \(\mathcal{C}^k\) on \(B^\delta\), and hence \(g(u)=L_\epsilon\mbk{u;\gbarv}+\mang{\bar{g},\bar{V}v(u)}\) is of class \(\mathcal{C}^k\) on \(B^\delta\). Consequently, \(f|_{\mathcal{M}}\) is of class \(\mathcal{C}^k\). 
%\end{proof}
\begin{theorem}\label{7the:5}
Let \(\bar{x}\) be a local minimizer of \(f\) and \(\mathcal{M}\mathrel{\mathop:}=\mcb{\bar{x}+\bar{U}u+\bar{V}v\mbk{u}\colon u\in B^\delta }\) be a \(\mathcal{C}^1\) fast track of \(f\) for some \(\delta\in\mathbb{R}_{+}\). 
If Assumptions \ref{7as:1},  \ref{7as:2}, \ref{7as:4}, \ref{7as:5} and \ref{7:as6} hold, 
then the subdifferential map \(\partial f\) is inner-semicontinuous at \(\bar{x}\) relative to \(\mathcal{M}\).
\end{theorem}
\begin{proof}
By the definition of inner-semicontinuity we need to show 
\begin{equation}\label{7eq:theo512}
\fa \bar{g}\in \partial f\mbk{\bar{x}},\ \fa x^k\overset{\mathcal{M}}{\rightarrow}\bar{x},\ \exists\,g^k\rightarrow \bar{g}\ \mathrm{ with }\ g^k\in \partial f\mbk{{x}^k}.
\end{equation}
We first show the following
\begin{equation}\label{7eq:theo58}
\fa \bar{g}\in \epri \partial f\mbk{\bar{x}},\ \fa x^k\overset{\mathcal{M}}{\rightarrow}\bar{x},\ \exists\,g^k\rightarrow \bar{g}\ \mathrm{ with }\ g^k\in \partial f\mbk{{x}^k}.
\end{equation} 
Let \(\bar{g}\) be an arbitrary element in \(\epri \partial f\mbk{\bar{x}}\). Set \(g(u)\mathrel{\mathop:}=f\mbk{G(u)}\) with \(G(u)\mathrel{\mathop:}=\bar{x}+\bar{U}u+\bar{V}v(u)\) and \(u\in B^\delta\). 
As \(\mathcal{M}\) is a \(\mathcal{C}^1\)-fast track, Lemma \ref{7lem:4} can be applied with \(k=1\).

From \eqref{7lem:4}(iv) it is easy to get 
\begin{equation}\label{7eq:theo51}
\hat{\partial}g\mbk{\bar{u}}=\mcb{\nabla g\mbk{\bar{u}}}=\mcb{\nabla L_\epsilon\mbk{\bar{u};\gbarv}+\mang{\bar{g},\bar{V}\nabla v(\bar{u})}},\ \fa \bar{u}\in B^\delta. 
\end{equation}
On the other hand, since \(g(u)=f\mbk{G(u)}\) and \(G(u)\) is of class \(\mathcal{C}^1\), the basic chain rule reveals 
\begin{equation}\label{7eq:theo52}
\hat{\partial}g\mbk{\bar{u}}\supset \nabla G\mbk{\bar{u}}^\top \hat{\partial}f\mbk{G(\bar{u})},\ \fa \bar{u}\in B^{\delta}.
\end{equation}
Consequently, we get from \eqref{7eq:theo51} and \eqref{7eq:theo52} that 
\begin{equation}
\nabla g\mbk{\bar{u}}=\nabla G\mbk{\bar{u}}^\top \hat{\partial}f\mbk{G(\bar{u})},\ \fa \bar{u}\in B^{\delta}.
\end{equation}
For all \(z\in\hat{\partial} f\mbk{G\mbk{\bar{u}}}\), we have 
\(\nabla g\mbk{\bar{u}}=\nabla G\mbk{\bar{u}}^\top z=\left[\bar{U}+\bar{V}\nabla v(\bar{u})\right]^\top\mbk{\bar{U}z_u+\bar{V}z_v}=\bar{U}^\top \bar{U}z_u+\nabla v\mbk{\bar{u}}^\top \bar{V}^\top\bar{V}z_v\).
As \(\bar{U}\) and \(\bar{V}\) are arbitrary basis matrices for \(\mathcal{U}\) and \(\mathcal{V}\), respectively, we can choose orthogonal ones to obtain \(\bar{U}^\top \bar{U}=I_m\) and \(\bar{V}^\top\bar{V}=I_p\) where \(I_m\) and \(I_n\) are identity matrices. 
Therefore, we get 
\begin{equation}\label{7eq:the53}
\nabla g\mbk{\bar{u}}=z_u+\nabla v\mbk{\bar{u}}^\top z_v,\ \fa z\in\hat{\partial} f\mbk{G\mbk{\bar{u}}},\ \bar{u}\in B^\delta.
\end{equation} 
We have 
\begin{equation}
\normt{G(\bar{u})-\bar{x}}=\normt{\bar{U}\bar{u}+\bar{V}v\mbk{\bar{u}}}\leq \normt{\bar{U}\bar{u}}+\normt{\bar{V}v\mbk{\bar{u}}}. 
\end{equation}
As \(\normt{v(u)}=o\mbk{\normt{u}}\), we can choose a sufficiently small number 
\(\zeta\in]0,\delta]\) such that \(\normt{G(\bar{u})-\bar{x}}\leq \hat{\delta}\) for all 
\(\bar{u}\in B^\zeta\mathrel{\mathop:}=\mcb{u\in\mathbb{R}^m\colon \normt{\bar{U}u} \leq \zeta}\), 
where \(\hat{\delta}\) is introduced in Assumption \ref{7as:4}. 
Consequently, Assumption \ref{7as:4} implies that 
\begin{equation}\label{7eq:the54}
\hat{\partial}f\mbk{G\mbk{\bar{u}}}=\partial f\mbk{G\mbk{\bar{u}}},\ \fa \bar{u}\in B^{\zeta}. 
\end{equation}
From \eqref{7eq:theo51}, \eqref{7eq:the53} and \eqref{7eq:the54} we see 
\begin{equation}\label{7eq:the57}
\nabla L_\epsilon\mbk{\bar{u};\bar{g}_v}+\nabla v(\bar{u})^\top \bar{V}^\top \bar{g}=z_u+\nabla v\mbk{\bar{u}}^\top z_v,\ \fa z\in{\partial} f\mbk{G\mbk{\bar{u}}},\ \bar{u}\in B^{\zeta}.
\end{equation} 
Next we show that there exists \(\bar{z}\in {\partial} f\mbk{G\mbk{\bar{u}}}\) such that \(\bar{z}_v=\bar{V}^\top \bar{g}\).
We express \(L_\epsilon\) as 
\begin{gather}
L_\epsilon\mbk{u;\gbarv}=\inf\limits_{v\in \mathbb{R}^p}h(u,v),\ \mathrm{ where }\\
h(u,v)=f\mbk{\bar{x}+\bar{U}u+\bar{V}v}-\mang{\bar{g},\bar{V}v}+\delta_{B^\epsilon}(v).
\end{gather} 
Using the same argument as in Theorem \ref{7th:2}(i) we can show that \(h(u,v)\) is proper, l.s.c. on \(\mathbb{R}^m\times\mathbb{R}^p\) and level bounded in $v$ locally uniformly in $u$. 
Thus, 10.13 of \cite{Rockafellar1998} can be applied to yield 
\begin{equation}\label{7eq:theo53}
\hat{\partial}L_{\epsilon}\mbk{\bar{u};\gbarv}\subset\bigcap\limits_{\bar{v}\in W\mbk{\bar{u};\gbarv}}\mcb{s\colon (s,0)\in\hat\partial h\mbk{\bar{u},\bar{v}}},\ \fa \bar{u}\in\dom L_{\epsilon}\mbk{{u};\gbarv}.
\end{equation}
Now consider any \(\bar{u}\in B^\tau\), as \(W\mbk{\bar{u};\gbarv}\bigcap\int B^\epsilon\not=\emptyset \), the smoothness of \(L_\epsilon\mbk{u;\gbarv}\) on \(B^\delta\) and \eqref{7eq:theo53} imply 
\begin{equation}\label{7eq:theo54}
\mbk{\nabla L_\epsilon\mbk{\bar{u};\gbarv},0}\in \hat{\partial}h\mbk{\bar{u},\bar{v}},\ \fa \bar{v}\in W\mbk{\bar{u};\gbarv}\bigcap\int B^\epsilon,\ \fa \bar{u}\in B^\tau.
\end{equation}
Define \(\tilde{g}(u,v)\mathrel{\mathop:}=f\mbk{\tilde{G}(u,v)}\), where \(\tilde{G}(u,v)\mathrel{\mathop:}=\bar{x}+\bar{U}u+\bar{V}v\).
From Proposition \ref{7pro:5} we get 
\begin{equation}\label{7eq:theo55}
 {\partial}h\mbk{\bar{u},\bar{v}}= {\partial}\tilde{g}\mbk{\bar{u},\bar{v}}-\mbk{0,\bar{V}^\top \bar{g}},\ \fa \bar{v}\in \int B^\epsilon.
\end{equation}
It is easy to see that the kernel of \(\nabla\tilde{G}\mbk{\bar{u},\bar{v}}^\top =\left[\bar{U}\ \bar{V}\right]^\top\) is \(\mcb{0}\) for all \(\mbk{\bar{u},\bar{v}}\in\mathbb{R}^m\times\mathbb{R}^p\). 
Hence we can apply the basic chain rule (see 10.6 of \cite{Rockafellar1998}) to obtain 
\begin{equation}\label{7eq:theo56}
\partial\tilde{g}\mbk{\bar{u},\bar{v}}\subset\nabla\tilde{G}\mbk{\bar{u},\bar{v}}^\top\partial f\mbk{\tilde{G}\mbk{\bar{u},\bar{v}}},\ \fa \mbk{\bar{u},\bar{v}}\in\mathbb{R}^m\times\mathbb{R}^p.
\end{equation}
From \eqref{7eq:theo54}, the fact that \(\hat{\partial}h\mbk{\bar{u},\bar{v}}\subset \partial h\mbk{\bar{u},\bar{v}}\), \eqref{7eq:theo55} and \eqref{7eq:theo56} it follows that 
\begin{gather}
\mbk{\nabla L_\epsilon\mbk{\bar{u};\gbarv},0}\in \nabla\tilde{G}\mbk{\bar{u},\bar{v}}^\top\partial f\mbk{\tilde{G}\mbk{\bar{u},\bar{v}}}-\mbk{0,\bar{V}^\top \bar{g}}\\
=\mbk{\bar{U}^\top \partial f\mbk{\tilde{G}\mbk{\bar{u},\bar{v}}},\bar{V}^\top \partial f\mbk{\tilde{G}\mbk{\bar{u},\bar{v}}}-\bar{V}^\top\bar{g}},\ \fa \bar{v}\in W\mbk{\bar{u};\gbarv}\bigcap\int B^\epsilon,\ \fa \bar{u}\in B^\tau.
\end{gather} 
Consequently, 
\begin{equation}\label{7eq:the55}
\bar{V}^\top\bar{g}\in \bar{V}^\top \partial f\mbk{\tilde{G}\mbk{\bar{u},\bar{v}}} \ \fa \bar{v}\in W\mbk{\bar{u};\gbarv}\bigcap\int B^\epsilon,\ \fa \bar{u}\in B^\tau.
\end{equation} 
From Assumption \ref{7as:5}, \(v\mbk{\bar{u}}\in \int B^\epsilon\). 
On the other hand the definition of \(v(\cdot)\) implies that \(v\mbk{\bar{u}}\in\bigcap\limits_{\bar{g}\in\epri \partial f\mbk{\bar{x}}} W\mbk{\bar{u};\bar{g}_v}\) for all \(\bar{u}\in B^\delta\).  
Therefore we can set the \(\bar{v}\) in \eqref{7eq:the55} to be \(v\mbk{\bar{u}}\) and get 
\begin{equation}\label{7eq:the56}
\bar{g}_v\in \bar{V}^\top \partial f\mbk{{G}\mbk{\bar{u}}}, \ \fa  \bar{u}\in B^\tau.
\end{equation}
Letting \(\gamma\mathrel{\mathop:}=\min\mcb{\tau,\zeta}\) then for all \(\bar{u}\in B^\gamma\mathrel{\mathop:}=\mcb{u\in\mathbb{R}^m\colon \normt{\bar{U}u} \leq \gamma}\), 
we have from \eqref{7eq:the56}, the orthogonality of \(\bar
V\) and \eqref{7eq:the57} that 
\begin{equation}\label{7eq:the59}
\nabla L_\epsilon\mbk{\bar{u};\bar{g}_v}=z_u,\ \fa z\in{\partial} f\mbk{G\mbk{\bar{u}}},\ \bar{u}\in B^{\gamma}.
\end{equation}
Now consider the sequence \(\mcb{x^k}\) in \eqref{7eq:theo58}.
As \(\mathcal{M}\) is a smooth manifold, each \(x\in\mathcal{M}\) corresponds to a \(u\) such that \(x=G(u)\) as long as \(u\in B^\delta\).
For each \(x^k\overset{\mathcal{M}}{\rightarrow}\bar{x}\), correspondingly there is \(u^k\overset{B^\delta}{\rightarrow}0\) with \(x^k=G\mbk{u^k}\).
Notice that \(\tau\leq\delta\) and therefore for \(k\) big enough, by \eqref{7eq:the56}, we can take an \(g^k\) to be an element in \(\partial f\mbk{x^k}\) such that 
\begin{equation}\label{7eq:the510}
\bar{g}_v=g^k_v.
\end{equation}
And from \eqref{7eq:the59} and the fact that \(\gamma\leq \zeta\leq\delta\) we can get for \(k\) big enough
\begin{equation}\label{7eq:the511}
\nabla L_\epsilon\mbk{u^k;\bar{g}_v}=g^k_u,
\end{equation}
From \eqref{7eq:the510},  \eqref{7eq:the511}, the smoothness of \(L_{\epsilon}\) and Theorem \ref{7the:4} we have
\begin{equation}
g^k=\bar{U} g^k_u+\bar{V} g^k_v\rightarrow \bar{U}\nabla L_\epsilon\mbk{0;\bar{g}_v}+\bar{V}\bar{g}_v=\bar{g}.
\end{equation}
We have showed \eqref{7eq:theo58}. 
Next we show \eqref{7eq:theo512}.
 To simplify notation we denote \(E\mathrel{\mathop:}=\partial f\mbk{\bar{x}}\setminus \epri \partial f\mbk{\bar{x}}\). 
 To show \eqref{7eq:theo512} based on \eqref{7eq:theo58} we only need to show the following
 \begin{equation}\label{7eq:theo514}
 \fa \bar{g}\in E,\ \fa x^k\overset{\mathcal{M}}{\rightarrow}\bar{x},\ \exists\,g^k\rightarrow \bar{g}\ \mathrm{ with }\ g^k\in \partial f\mbk{{x}^k}.
 \end{equation}
 From Assumption \ref{7:as6} we see that 
 \begin{equation}\label{7eq:the515}
\fa g_v\in P_{\mathcal{V}}\mbk{E},\ \fa x^k\overset{\mathcal{M}}{\rightarrow}\bar{x},\ \exists\,g^k_v\rightarrow {g}_v\text{ with } g^k\in \partial f\mbk{x^k}.
\end{equation}
Notice from \eqref{7eq:the59} and \eqref{7eq:the511} we can set the \(g^k\) in \eqref{7eq:theo514} to be \(\bar{U}\nabla L_\epsilon\mbk{u^k;\bar{g}_v}+\bar{V}g^k_v\) where \(g^k_v\) is introduced in \eqref{7eq:the515} and thus \eqref{7eq:theo514} follows from the smoothness of \(L_{\epsilon}\), Theorem \ref{7the:4} and Corollary \ref{7cor:1}.
\end{proof}
\begin{asump}\label{7as:7}
There exists a real number \(\bar \delta\) such that at every point in \(B\mbk{\bar{x},\bar\delta}\cap \mathcal{M}\), \(f\) is subdifferentially regular and has a subgradient.
\end{asump}
\begin{theorem}
Under Assumptions \ref{7as:1}, \ref{7as:2}, \ref{7as:5}, \ref{7:as6}, and \ref{7as:7}, if \(\bar{x}\) is a local minimizer of \(f\) and \(\mathcal{M}\mathrel{\mathop:}=\mcb{\bar{x}+\bar{U}u+\bar{V}v\mbk{u}\colon\bar{U} u\in B}\) is a \(\mathcal{C}^k\) fast track of \(f\), 
then \(f\) is \(\mathcal{C}^k\)-partly smooth at \(\bar{x}\) relative to \(\mathcal{M}\).
\end{theorem}
\begin{proof}
Suppose \(\bar{x}\) is a local minimizer of \(f\) and \(\mathcal{M}\) is a \(\mathcal{C}^k\) fast track of \(f\), then by Definition \ref{7def:10}, 
\(v(u)\in \bigcap\limits_{\bar{g}\in\epri \partial f\mbk{\bar{x}}} W\mbk{u;\bar{g}_v}\) and hence \(v(u)\) is a special case of the \(v\mbk{u;\bar{g}_v}\) in
 Lemma \ref{7lem:4}. 
Furthermore, Lemma \ref{7lem:4} can be applied and \(\mathcal{M}\) is a special case of the \(\mathcal{C}^k\)-smooth manifold in Lemma \ref{7lem:4}(i).   
By Lemma \ref{7lem:4}(iv) \(f|_{\mathcal{M}}\) is of class \(\mathcal{C}^k\) and thus \(f\) satisfies property (i) in Definition \ref{7def:13}. 
Property (ii) is also satisfied because of Assumption \ref{7as:7}. 
Taking orthogonal complements of the two sides in \eqref{7lem:4} gives property (iii) in Definition \ref{7def:13}. 
Next we show property (iv). 
It suffices to show that \(\partial f\) is inner-semicontinuous at \(\bar{x}\) relative to \(\mathcal{M}\) as \(\partial f\) is already outer-semicontinuous at at \(\bar{x}\).
To apply Theorem \ref{7the:5} we only need to verify that Assumption \ref{7as:4} holds, which is immediate from Assumption \ref{7as:7}.
\end{proof}
\bibliographystyle{spmpsci}
\bibliography{F://Academic/JournalArticles/ReferenceDatabase}
\end{document}